\documentclass[12pt]{amsart}
\usepackage{amsmath}
\usepackage{tikz}
\usetikzlibrary{calc} 
\usepackage{amsthm}
\usepackage{amssymb}
\usepackage{graphicx}
\usepackage{blindtext}
\usepackage{hyperref}
\numberwithin{equation}{section}

\usepackage[a4paper,top=4cm,bottom=3.5 cm,left=3.5 cm,right=3.5cm]{geometry}

\makeatletter
\@namedef{subjclassname@2020}{\textup{2020} Mathematics Subject Classification}
\makeatother

\title[Torsional Rigidity on Metric Graphs with $\delta$-Conditions]
{Torsional Rigidity on Metric Graphs with Delta-Vertex Conditions}

\author[S.~Özcan]{Sedef Özcan} 
\author[M.~Täufer]{Matthias Täufer}
\address{Sedef Özcan, Dokuz Eylül University, Faculty of Science, Department of Mathematics, Izmir, Turkey}
\address{Matthias Täufer, Fakult\"at Mathematik und Informatik, Fern\-Universit\"at in Hagen, D-58084 Hagen, Germany}

\thanks{\emph{Acknowledgements.} The authors would like to thank Delio Mugnolo for suggesting this research direction and for numerous helpful discussions. 
S.Ö. acknowledges support by COST Action 18232 MAT-DYN-NET, supported by COST (European Cooperation in Science and Technology), www.cost.eu.
M.T. was supported by Flexibler Fonds Nachwuchs of FernUniversität in Hagen.
}



\newtheorem{theorem}{Theorem}[section]
\newtheorem{notation}{Notation}[section]
\newtheorem{lemma}[theorem]{Lemma}
\newtheorem{proposition}[theorem]{Proposition}

\newtheorem{corollary}[theorem]{Corollary}
\newtheorem{definition}[theorem]{Definition}
\newtheorem{example}[theorem]{Example}
\newtheorem{remark}[theorem]{Remark}
\date{\today}

\begin{document}

\begin{abstract}
	We investigate the \emph{torsion function} or \emph{landscape function} and its integral, the torsional rigidity, of Laplacians on metric graphs subject to $\delta$-vertex conditions. 
	A variational characterization of torsional rigidity and Hadamard-type formulas are obtained, enabling the derivation of surgical principles. 
	We use these principles to prove upper and lower bounds on the torsional rigidity and identify graphs maximizing and minimizing torsional rigidity among classes of graphs. 
	We also investigate the question of positivity of the torsion function and reduce it to positivity of the spectrum of a particular discrete, weighted Laplacian.
	Additionally, we explore potential manifestations of Kohler-Jobin-type inequalities in the context of $\delta$-vertex conditions.
\end{abstract}

\maketitle

\section{Introduction}
The torsional rigidity is a spectral-geometric quantity, closely related to the lowest eigenvalue of the Laplacian.
It is defined as the $L^1$-norm of the so-called \emph{torsion function} which in recent years has gained notoriety as the so-called \emph{landscape function}, allowing to deduce pointwise bounds on eigenfunctions.
In this work, we investigate the torsion function and torsional rigidity on metric graphs with $\delta$-vertex conditions.

In the past decades, there has been extensive research on differential operators on metric graphs, also referred to as quantum graphs.
They serve as one-dimensional models for networks or quasi-one-dimensional structures such as nanotubes or waveguides. A comprehensive overview can be found in \cite[Section 7.6]{berk3}.
A substantial body of literature has emerged on the spectral theory of self-adjoint Schrödinger operators on metric graphs, see~\cite{berk3,kost1,kost2,kuch1,kuch2} for a non-exhaustive list.
Recently, there has been a particular focus on {spectral geometry on metric graphs, that is} the question how the geometry or topology of the graph will affect Laplacian eigenvalues.
A quantity, which essentially behaves as a counterpart to the lowest Laplacian eigenvalue is the torsional rigidity.
Indeed, historically, many questions in the spectral geometry of domains such as the Faber-Krahn inequality~\cite{Faber-23, Krahn-25} have experienced parallel developments to analogous developments for the torsional rigidity.
Nevertheless, investigations of the torsional rigidity on metric graphs, which are nowadays also often used as toy models to build intuition before proceeding to tackle nontrivial questions on domains, have been limited so far, with the only contribution known to us in~\cite{pluem,colla}.

Therefore, in this article, we seek to advance the theory of torsional rigidity and in particular its connection with the geometry of the graph.
Our novelty is that we allow for $\delta$-vertex conditions at vertices which will lead to a much richer theory.
% Adhering to the terminology established in~\cite{pluem}, we will refer
More precisely the \emph{torsion function} of a metric graph $\mathcal{G}$ with vertex set $\mathsf{V}$ and edge set $\mathsf{E}$ is the unique solution $\upsilon$ to the elliptic problem
\begin{equation}
	\begin{cases}
		-\Delta \upsilon(x) =1,& x \in \mathcal{G},\\
		\sum_{\mathsf{e} \in\mathsf{E}_{\mathsf{v}} }\frac{\partial \upsilon_\mathsf{e}}{\partial n}(\mathsf{v})+\alpha_{\mathsf{v}} \upsilon(\mathsf{v} )=0,&  \mathsf{v} \in \mathsf{V},
	\end{cases}
\end{equation}
{where $\alpha_{\mathsf{v}}$ is a real number -- for simplicity, one can assume all $\alpha_{\mathsf{v}}$ strictly positive even though this can be relaxed --} for each vertex $\mathsf{v}\in \mathsf{V}$, called \emph{strength of the vertex} and $\mathsf{E}_{\mathsf{v}}$ is the set of egdes connected to the vertex $\mathsf{v}$.
The \emph{torsional rigidity} is the $L^{1}$-norm of the torsion function
\begin{equation}
	T(\mathcal{G}):= \lVert \upsilon \rVert_{L^{1}(\mathcal{G})}.
\end{equation}
Initially, it arose in mechanics as the $L^{1}$-norm of the solution $\upsilon$ of
$$\begin{cases}
	-\Delta \upsilon(x) =1, &x \in \Omega,\\
	\upsilon(z)=0, & z \in \partial \Omega,
\end{cases}$$
for open, bounded domains $\Omega\subset\mathbb{R}^{2}$.
However, P{\'o}lya identified it as a geometric constant depending on the shape and size of $\Omega$  in~\cite{pol1}. He demonstrated that among all open bounded domains, the circular domains maximise torsional rigidity.

In recent years, the torsion function has also made a renaissance, on the one hand in more general situations such as for non-local operators~\cite{MazonT-23} as the \emph{landscape function} leading to pointwise bounds on eigenfunctions in quantum systems~\cite{Steinerberger-17, ArnoldDFJM-19, HarrellM-20, ken2, mug2}, and yielding new insights in phenomena in mathematical physics such as Anderson localization~\cite{FilocheM-12}.

In~\cite{pluem}, Mugnolo and Plümer developed a theory of torsional rigidity for the Laplacian on metric graphs with at least one Dirichlet vertex.
Note that the assumption of at least one Dirichlet vertex is necessary for otherwise, the Laplacian will not be invertible and the torsion function will not be defined, let alone positive.
In the presence of $\delta$-vertex conditions we first have to investigate under which conditions on strengths, the existence and positivity of the torsion function can be ensured.
Indeed, we shall see that the question of existence of a positive torsion function boils down to positivity of the spectrum of a suitable weighted discrete Laplacian with $\delta$-vertex conditions.

Bounds on the torsional rigidity of the Laplacian on domains with Robin boundary conditions have been investigated in~\cite{ciac,kost1}.
For quantum graphs, we will likewise establish sharp lower and upper bounds on the torsional rigidity on all quantum graphs with a specified total length.
In particular, it will turn out that flower graphs and path graphs take roles of minimizers and maximizers in certain situations but one has to be careful where one puts the strengths.

Our findings might also have interesting consequences for and develop some new insight into possible Kohler-Jobin-type estimates which we discuss in Section~\ref{Koh-job}.

This article is structured as follows: In Section~\ref{pre}, we define the torsional rigidity on quantum graphs for Laplacian with $\delta$-vertex conditions.
In Section~\ref{secvarchar}, we derive a variational characterization of torsional rigidity.
In Section~\ref{sec:Dirichlet}, we extend the theory to mixed Dirichlet and $\delta$-vertex conditions.
In Section~\ref{seclin} we investigate the relationship between vertex strengths and the positivity of the torsion function.
In particular, we reduce the question of positivity the torsion function to the question of positivity of the spectrum of a weighted \emph{discrete Laplace operator}.
We then analyze in Section~\ref{sec:hadamard} how the torsional rigidity responds to perturbations in edge lengths and strengths.
In Section~\ref{surpri} we study the behaviour of the torsional rigidity under surgery principles for the torsional rigidity on quantum graph operations such as lengthening or shortening an edge, connecting or cutting vertices, adding a new graph to the graph from one of its vertices, etc.
In Section~\ref{est} we establish lower and upper bounds on the torsional rigidity.
We conclude the article with a discussion on possible forms of Kohler-Jobin-type inequalities on torsional rigidity and the first Laplacian eigenvalue in the presence of $\delta$-vertex conditions underscoring that estimates on the torsional rigidity are intimately related with spectral estimates.

\section{Torsional Rigidity}\label{pre}

\begin{definition}
	A \emph{metric graph} $\mathcal{G}$ is a finite collection of intervals, topologically glued together at their endpoints according to the structure of a combinatorial graph $\mathsf{G} = (\mathsf{V}, \mathsf{E})$ with finite vertex and edge set $\mathsf{V}$ and $\mathsf{E}$, respectively.
	Every edge $\mathsf{e} \in \mathsf{E}$ is identified with an interval $[0, \ell_{\mathsf{e}}]$ where $\ell_{\mathsf{e}} \in (0, \infty)$ is the \emph{edge length}.
\end{definition}

Metric graphs are also called "topological networks"~\cite{nica}. 
They inherit the shortest-math metric and the Lebesgue measure, see~\cite{mug3} for more details.
In the recent years, also \emph{infinite metric graphs} have been studied intensively in which the case the edge set and vertex sets are no longer finite, see e.g.~\cite{KostenkoN-19,KostenkoMN-22} and references therein.
However, infinite metric graphs can introduce new phenomena where already the definition of a Laplace operator is a touchy issue and even the question of discreteness of the spectrum, or its strict positivity is more subtle~\cite{DuefelKennedyMugnoloPluemerTaeufer-22, KennedyMugnoloTaeufer-2024_preprint}.
Therefore, in this article, we always assume in this article that the edge set $\mathsf{E}$ is finite.
Let us start with some basic definitions:

\begin{definition}\label{assump}
	 A metric graph $\mathcal{G}$ is \emph{connected} if for any $x, y \in \mathcal{G}$, there is a continuous path $\mathcal{P}$ connecting $x$ and $y$.
	 \\
For a vertex $\mathsf{v} \in \mathsf{V} $, let $\mathsf{E}_{\mathsf{v}}$ denote the edges adjacent to it and $\operatorname{deg}_{\mathcal{G}}(\mathsf{v} ) := \lvert \mathsf{E}_{\mathsf{v}} \rvert$ its \emph{degree}. 
	\\
The \emph{shortest-path metric} on $\mathcal{G}$ is denoted by $\operatorname{dist}_{\mathcal{G}}:\mathcal{G} \times \mathcal{G}\rightarrow [0,\infty)$. 
	\\
The \emph{total length} $\lvert \mathcal{G} \rvert$ of $\mathcal{G}$ is the sum of all its edge lengths.

\end{definition}

Throughout this article, for a function $u \colon \mathcal{G} \to \mathbb{C} $, $u_{\mathsf{e}}$ denotes its restriction to the edge $\mathsf{e}$.
We will often need derivatives on edges at their end points and adhere to the following convention:
\begin{notation} 
We write $\frac{\partial u_{\mathsf{e}} }{\partial n}(\mathsf{v} )$ for the value of the derivative of $u$ on the edge $\mathsf{e}$ at its endpoint $\mathsf{v} $, pointing into $\mathsf{v} $.
\end{notation}

We define the Sobolev space
\begin{equation*}
	H^{1}(\mathcal{G})
	=
	\left\lbrace  u=(u_{\mathsf{e}})_{\mathsf{e} \in \mathsf{E}} \in \bigoplus_{\mathsf{e}\in \mathsf{E}} H^{1}(0,\ell_{\mathsf{e}}) \ \Big\vert \ \text{$u$ continuous in every $\mathsf{v} \in \mathsf{V}$}  \right\rbrace, 
\end{equation*}
and for $\alpha \in \mathbb{R}^{\mathsf{V}}$ the quadratic form $h_{\alpha} = h_{\mathcal{G}, \alpha}$ with domain $H^1(\mathcal{G})$ as
\begin{equation*}
	h_{\alpha}(u)
	=
	\int_{\mathcal{G}}\lvert u'(x) \rvert^{2}\mathrm{d} x+\sum_{\mathsf{v}  \in \mathsf{V}}\alpha_{\mathsf{v} } \lvert u(\mathsf{v} ) \rvert^{2}.
\end{equation*}
This gives rise to a self-adjoint operator in $L^2(\mathcal{G})$, called the \emph{Laplacian with $\delta$-vertex conditions}, $\Delta_{\mathcal{G}, \alpha}$.
Its domain consists functions which are edgewise in the $H^2(0, \ell_{\mathsf{e}})$ Sobolov space and satisfy the vertex conditions
\begin{equation}
\label{vertexcond_Dir}
\begin{cases}
	\text{$u$ is continuous } \text{at every $\mathsf{v} \in \mathsf{V}$},
	\\
	\sum_{\mathsf{e} \in \mathsf{E}_{\mathsf{v} }}\frac{\partial u_{\mathsf{e}} }{\partial n}(\mathsf{v} )+\alpha_{\mathsf{v} } u(\mathsf{v} )=0
	
	\text{ for all $\mathsf{v}  \in \mathsf{V}$}.
\end{cases}
\end{equation}
If $\alpha_{\mathsf{v}} = 0$, we recover the usual \emph{standard vertex condition} or \emph{Kirchhoff-Neumann condition} at $\mathsf{v}$.
Before defining the torsion function $\upsilon$ let us ensure its existence.	
 
\begin{proposition}\label{postor}
Let $\mathcal{G}$ be a connected metric graph and $\alpha \in [0, \infty)^{\mathsf{V}}$.
If at least one $\alpha_{\mathsf{v}}$ is strictly positive, then the equation
\begin{equation}\label{torrob}
	-\Delta_{\mathcal{G}, \alpha} \upsilon =1.
\end{equation}
has a unique strictly positive solution in $H^1(\mathcal{G})$.
\end{proposition}

\begin{definition}\label{torfunc}
Let $\mathcal{G}$ be a connected metric graph.
The \emph{torsion function} $\upsilon$ of $\mathcal{G}$ is the unique solution $\upsilon$ of equation~\eqref{torrob}.
The \emph{torsional rigidity} of $\mathcal{G}$ with vertex strengths $\alpha$ is 
\begin{equation*}
	T(\mathcal{G}, \alpha)=\int_{\mathcal{G}} \upsilon(x) \mathrm{d} x.
\end{equation*}
\end{definition}

\begin{remark}
	Implicitly, we assume in Definition~\ref{torfunc} that $\alpha$ is chosen in such a way that the torsion function $\upsilon$ exists.
	Proposition~\ref{postor} provides one way to ensure this, but its assumptions are not always necessary as we shall see in Section~\ref{seclin}.
\end{remark}

\begin{proof}[Proof of Proposition~\ref{postor}]
 We may assume that there is only one positive strength $\alpha_{\mathsf{v} }$ at the vertex $\mathsf{v} $.
 We will prove that the first eigenvalue $\lambda_{1}$ of $-\Delta_{\mathcal{G},\alpha}$ is positive by using its variational characterization:
	\begin{equation}\label{firsteigenvalue}
	\lambda_{1}
	=
	\inf_{u \in H^{1}(\mathcal{G})} \frac{h_{\alpha}(u)}{\lVert u \rVert_{L^{2}(\mathcal{G})}^{2}}=\inf_{\substack{u \in H^{1}(\mathcal{G}) \\  \lVert u \rVert_{L^{2}(\mathcal{G})}=1}} h_{\alpha}(u).
 	\end{equation} 
	Let $u \in H^{1}(\mathcal{G})$. 
	Choose $\kappa>0$ with $\alpha_{\mathsf{v} }>\kappa \lvert \mathcal{G} \rvert$ and  let $x_{0} \in \mathcal{G}$ such that $\lvert u(x_{0}) \rvert> 1/ \sqrt{\lvert \mathcal{G} \rvert}$. 
	Without loss of generality, we may assume $u(x_0) > 0$.
	We distinguish two cases: either $\lvert u(\mathsf{v} ) \rvert\geq \frac{\kappa}{\alpha_{\mathsf{v} }}$ or $\lvert u(\mathsf{v} ) \rvert<\frac{\kappa}{\alpha_{\mathsf{v} }}$. 
	In the former case we have $h_{\alpha}(u)>\kappa>0$. 
	In the latter, we observe by connectedness
 \begin{align*}
	u(x_{0})-u(\mathsf{v} )
	&=
	\int_{x_{0}}^{\mathsf{v} } u'(t)\mathrm{d} t
	\leq 
	\sqrt{ \lvert x_{0}-\mathsf{v} \rvert}\int_{x_{0}}^{\mathsf{v} } \lvert u'(t) \rvert^{2}\mathrm{d} t
	\\
	&\leq 
	\sqrt{\lvert x_{0}-\mathsf{v} \rvert}\int_{\mathcal{G}} \lvert u'(t) \rvert^{2}\mathrm{d} t.
 \end{align*}
 Since $\frac{1}{\sqrt{\lvert \mathcal{G} \rvert}}-\sqrt{\frac{\kappa}{\alpha_{\mathsf{v} }}}<u(x_{0})-u(\mathsf{v} )$, we have
 $$0< \frac{1}{\sqrt{\lvert x_{0}-\mathsf{v} \rvert}}\left( \frac{1}{\sqrt{\lvert \mathcal{G} \rvert}}-\sqrt{\frac{\kappa}{\alpha_{\mathsf{v} }}} \right)< \int_{\mathcal{G}} \lvert u'(t) \rvert^{2}\mathrm{d} t.$$
 We showed for every $u\in H^{1}(\mathcal{G})$
 \begin{equation*}
h_{\alpha}(u)> \min \left\lbrace \kappa,  \frac{1}{\sqrt{\lvert \mathcal{G} \rvert}}\left( \frac{1}{\sqrt{\lvert \mathcal{G} \rvert}}-\sqrt{\frac{\kappa}{\alpha_{\mathsf{v} }}} \right) \right\rbrace >0.
 \end{equation*}
 Laplace operators on metric graphs with $\delta$-vertex conditions are generators of positive semigroups \cite[Theorem 6.71]{mug1}. 
 Since the infimum of the spectrum of -$\Delta_{\mathcal{G},\alpha}$ is positive, this implies that there exists $u\geq 0$ with $-\Delta_{\mathcal{G},\alpha}u=\boldsymbol{1}$ \cite[Corollary 7.4]{batka}. 
\end{proof}
In Proposition~\ref{postor}, we proved that if all strengths are non-negative, with at least one of them being positive, then the torsion function uniquely exists and is positive. Let us show via the following examples that in general this condition cannot be readily dropped.
 
 The first example is a graph with $\delta$-vertex conditions where the sum of the strengths is zero due to the presence of negative strengths. We demonstrate that, in this case, the corresponding torsion function takes negative values.

\begin{example} 
{
\label{exa:2.6}
Let $\mathcal{G}$ be the metric graph given by an interval $[-1,1]$ with strengths $+1$ at both end points and strength $-2$ in the middle, cf. Figure~\ref{fig:2.6}.
Its torsion function is 
\[
\upsilon(x) = -\frac{x^{2}}{2} + \frac{1}{4} \lvert x \rvert - 1.
\] 
In particular, \(\upsilon(0) = -1 < 0\). 

This example can even be refined by replacing $\alpha_0$ by $-2 + \epsilon$. Then, even the sum of strengths will be positive but for sufficiently small $\epsilon > 0$, the value of $\upsilon(0)$ will remain negative.
}

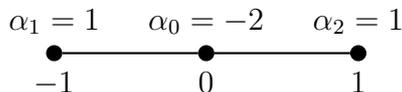
\begin{figure}[ht]
		\begin{tikzpicture}
			% İkinci graf
			
			\node[draw, circle, inner sep=2pt, fill, label=above:{$\alpha_{2}= 1$}, label=below:{$1$}] (A) at (6, 2) {};
			\node[draw, circle, inner sep=2pt, fill, label=above:{$\alpha_{0}=-2$}, label=below:{$0$}] (B) at (4, 2) {};
			\node[draw, circle, inner sep=2pt, fill, label=above:{$\alpha_{1}=1$}, label=below:{$-1$}] (C) at (2, 2) {};
			
			% Kenarlar
			\draw[thick] (A) -- (B);
			\draw[thick] (B) -- (C);
			
		\end{tikzpicture}
	\caption{The metric graph from Example~\ref{exa:2.6}. 
	The sum of strengths is zero but the torsion function is not positive definite.}
	\label{fig:2.6}
\end{figure}	

\end{example}

In the second example, we go further and show that even though the sum of the strengths is positive, the presence of negative strengths prevents us from guaranteeing the existence of a torsion function.

\begin{example}
	\label{exa:2.7}
{
Let $\mathcal{G}$ be the metric graph given by an interval $[-1,1]$ with strengths $+2$ at both end points and strength $-3$ in the middle, cf. Figure~\ref{fig:2.7}.

The sum of strenghts is $+1$ and thus comfortably positive but the infimum of the spectrum of the Laplacian $- \Delta_{\mathcal{G}, \alpha}$ is negative as can be seen by applying the quadratic form $h_\alpha$ to the test function $f(x) = 1 - \lvert x \rvert$.
This implies that the first eigenvalue must be negative, as per~\eqref{firsteigenvalue}. 
Therefore, we cannot be certain that there is a solution to the torsion equation~\eqref{torrob} as higher eigenvalues might happen to coincide with zero. Note that this example can be modified by replacing the strengths at the boundary vertices by any \(C > 0\).
}

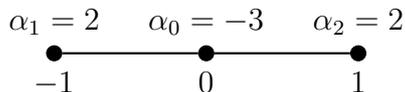
\begin{figure}[ht]
		\begin{tikzpicture}
			% İkinci graf
			
			\node[draw, circle, inner sep=2pt, fill, label=above:{$\alpha_{2}= 2$}, label=below:{$1$}] (A) at (6, 2) {};
			\node[draw, circle, inner sep=2pt, fill, label=above:{$\alpha_{0}=-3$}, label=below:{$0$}] (B) at (4, 2) {};
			\node[draw, circle, inner sep=2pt, fill, label=above:{$\alpha_{1}=2$}, label=below:{$-1$}] (C) at (2, 2) {};
			
			% Kenarlar
			\draw[thick] (A) -- (B);
			\draw[thick] (B) -- (C);
			
		\end{tikzpicture}
	\caption{The metric graph from Example~\ref{exa:2.7}. 
	The sum of strengths is positive but the infimum of the spectrum can be negative.}
	\label{fig:2.7}
\end{figure}

\end{example}

\section{Variational characterization}\label{secvarchar}
The following theorem, the main result of this section, is a variational cha\-racterization for the torsional rigidity.

\begin{theorem}\label{thm:varchar}
Let $\mathcal{G}$ be a connected metric graph and let $\alpha \in [0, \infty)^{\mathsf{V}}$ such that at least one $\alpha_{\mathsf{v}}$ is strictly positive.
Then
\begin{equation}\label{varchar}
	T(\mathcal{G}, \alpha)=\sup_{u\in H^{1}(\mathcal{G})}\frac{(\int_{\mathcal{G}}\lvert u \rvert \mathrm{d} x)^{2} }{h_{\alpha}(u)}=\max_{u\in H^{1}(\mathcal{G})}\frac{\lVert u \rVert^{2}_{L^{1}(\mathcal{G})} }{h_{\alpha}(u)}.
\end{equation}
Furthermore, $u \in H^1(\mathcal{G})$ maximizes~\eqref{varchar} if and only if it is a scalar multiple of the torsion function.
\end{theorem}

	Indeed, we shall see that the condition that all $\alpha_{\mathsf{v}}$ are positive and at least one strictly positive can be relaxed. 
	In the proof of Theorem~\ref{thm:varchar}, we only use have existence and positivity of the torsion function.

The quotient 
\[
	\frac{\lVert u \rVert^{2}_{L^{1}(\mathcal{G})} }{h_{\alpha}(u)}
\]
is called \emph{P\'olya quotient}, in analogy to the well-known \emph{Rayleigh quotient}
\[
	\frac{h_{\alpha}(u)}{\lVert u \rVert^{2}_{L^{2}(\mathcal{G})} }
\]
from the variational characterization of eigenvalues. Note the $L^{1}$-norm in the P\'olya quotient whereas the Rayleigh quotient has an $L^{2}$-norm.

\begin{proof}[Proof of Theorem~\ref{thm:varchar}]
Let $I:H^{1}(\mathcal{G})\rightarrow \mathbb{R}$ be defined as    
\begin{equation}\label{func}
	I(u)=\int_{\mathcal{G}}u\mathrm{d} x-\frac{1}{2}\left( \int_{\mathcal{G}}\lvert u' \rvert^{2}\mathrm{d} x +\sum_{\mathsf{v}\in \mathsf{V}}\alpha_{\mathsf{v}}\lvert u(\mathsf{v}) \rvert^{2}\right).
\end{equation}
First, note that $I:H^{1}(\mathcal{G})\rightarrow \mathbb{R}$ is strictly concave and continuous. We will prove that $I$ has a unique maximizer by showing that it is anti-coercive, i.e. $I(u) \rightarrow -\infty $ as $\lVert u \rVert_{H^{1}(\mathcal{G})}\rightarrow \infty$. 

Let $\tilde{\mathcal{G}}$ be the metric graph obtained by doubling each edge of $\mathcal{G}$ and $\Delta_{\tilde{\mathcal{G}},_{2\alpha}}$ be the Laplacian on $\tilde{\mathcal{G}}$ with $\delta$-vertex conditions twice of strengths of $\Delta_{\mathcal{G},_{\alpha}}$. For $u \in H^{1}(\mathcal{G})$, let \(\tilde{u}\) be the function on $\tilde{\mathcal{G}}$ obtained by copying each edge function of \(u\). Then $\tilde{u} \in H^{1}(\tilde{\mathcal{G}})$ and
\begin{equation*}
\int_{\mathcal{G}} u^{2}(x)\mathrm{d} x=\frac{1}{2}\int_{\tilde{\mathcal{G}}}\tilde{u}^{2}(x)\mathrm{d} x, \quad \int_{\mathcal{G}} u'(x)^{2}\mathrm{d} x=\frac{1}{2}\int_{\tilde{\mathcal{G}}}\tilde{u}'(x)^{2}\mathrm{d} x.
\end{equation*}
As we can cover the entirety of $\tilde{\mathcal{G}}$ by a single loop of length $2\lvert \mathcal{G} \rvert$, we have
\begin{equation*}
	\lVert u \rVert_{L^{2}(\mathcal{G})}=\frac{1}{2}\int_{0}^{2\lvert \mathcal{G} \rvert}\tilde{u}^{2}(x)\mathrm{d} x,
\end{equation*}
see~\cite{karr} for more detailed information about Eulerian cycle technique.
Let $\mathsf{v}_{0}$ be the vertex of $\mathcal{G}$ whose strength $\alpha_{0}$ is positive. We can choose the loop covering $\mathcal{G}$ starting at $\mathsf{v}_{0}$. Then
\begin{align*}
	\int_{0}^{2\lvert \mathcal{G} \rvert}\tilde{u}^{2}(x)\mathrm{d} x=&\int_{0}^{2\lvert \mathcal{G} \rvert} \left( \int_{\mathsf{v}_{0}}^{x}\tilde{u}'(t)\mathrm{d} t+\tilde{u}(\mathsf{v}_{0}) \right)^{2}\mathrm{d} x\\
	\leq& 2  \int_{0}^{2\lvert \mathcal{G} \rvert} \left( \int_{\mathsf{v}_{0}}^{x}\tilde{u}'(t)\mathrm{d} t \right)^{2}+\tilde{u}^{2}(\mathsf{v}_{0})\mathrm{d} x\\
	\leq& 4\lvert \mathcal{G} \rvert  \int_{0}^{2\lvert \mathcal{G} \rvert}  \int_{0}^{2\lvert \mathcal{G} \rvert} \tilde{u}'(t)^{2}\mathrm{d} t\mathrm{d} x+2\lvert \mathcal{G} \rvert\tilde{u}^{2}(\mathsf{v}_{0})\\
	=&8\lvert \mathcal{G} \rvert \lVert \tilde{u}'\rVert^{2}_{L^{2}(\tilde{\mathcal{G}})}+2\lvert \mathcal{G} \rvert\tilde{u}^{2}(\mathsf{v}_{0}),
\end{align*}
where in the third step we used Jensen's inequality. This implies 
\begin{equation*}
	\lVert u \rVert_{H^{1}(\mathcal{G})}^{2}\leq (8\lvert \mathcal{G} \rvert+1) \lVert u'\rVert_{L^{2}(\mathcal{G})}+\lvert \mathcal{G} \rvert u^{2}(\mathsf{v}_{0}).
\end{equation*}
Setting $C=\max\left\lbrace 8\lvert \mathcal{G} \rvert+1,\frac{\lvert \mathcal{G} \rvert}{\alpha_{0}}\right\rbrace $, we find
\begin{align*}
	\lVert u \rVert_{H^{1}(\mathcal{G})}^{2}\leq C \left( \lVert u' \rVert_{L^{2}(\mathcal{G})}+\alpha_{0}u^{2}(\mathsf{v}_{0})\right)
	\leq C h_{\alpha}(u)
\end{align*}
which gives 
\begin{equation*}
I(u)\leq \lvert \mathcal{G} \rvert \lVert u \rVert_{H^{1}(\mathcal{G})}- \frac{1}{2C} \lVert u \rVert_{H^{1}(\mathcal{G})}^{2}.
\end{equation*}
This proves anti-coercivity of $I$ and thus existence of a unique maximizer.

 Let $\upsilon \in H^{1}(\mathcal{G})$ be this unique maximizer. 
 We next prove that $\upsilon$ is the unique weak solution of  
\begin{equation*}
	-\Delta_{\mathcal{G}, \alpha} u =1.
\end{equation*}
 Fix any $u\in H^{1}(\mathcal{G})$ and write $i(\tau)=I(\upsilon+\tau u)$, $\tau \in \mathbb{R}$. Since $\upsilon+\tau u \in H^{1}(\mathcal{G})$ for each $\tau$, the scalar function has maximum at 0, thus $\frac{di}{d\tau}(0)=0$, provided this derivative exists. But 
\begin{equation*}
	0=i'(0)=\int_{\mathcal{G}} (u-\upsilon'u')\mathrm{d} x-\sum_{\mathsf{v} \in \mathsf{V}}\alpha_{\mathsf{v}}\upsilon(\mathsf{v})u(\mathsf{v}),
\end{equation*}
which means that $\upsilon$ is a weak solution of the equation $-\upsilon''=1$ with $\delta$-vertex conditions, i.e. the torsion function.

Next, note that for $A,B>0$,
	\begin{equation}\label{auxlem}
		At-B\frac{t^2}{2}\leq \frac{1}{2} \frac{A^2}{B}
	\end{equation}
	with  equality if and only if $t= \frac{A}{B}$, cf.~\cite[Lemma 2.1]{bras}.
In the following argument, we apply~\eqref{auxlem} for a fixed $0 \leq u \in H^{1}(\mathcal{G})$. Note that $I(u)\leq I(\lvert u \rvert)$, whence it suffices to consider positive functions in $H^{1}(\mathcal{G})$.
\begin{align*}
	I(u)
	&=
	\int_{\mathcal{G}}u \mathrm{d} x-\frac{1}{2}\left( \int_{\mathcal{G}}\lvert u' \rvert^{2}\mathrm{d} x+\sum_{\mathsf{v} \in \mathsf{V}}\alpha_{\mathsf{v}} u(\mathsf{v})^{2}\right) \\
	&\leq 
	\max_{\lambda  \in \mathbb{R}}\left[ \lambda \int_{\mathcal{G}}u\mathrm{d} x-\frac{\lambda^{2}}{2}\left( \int_{\mathcal{G}}\lvert u' \rvert^{2}\mathrm{d} x+\sum_{\mathsf{v} \in \mathsf{V}}\alpha_{\mathsf{v}} u(\mathsf{v})^{2}\right) \right] 
	\\
	&=
	\max_{\lambda  \in \mathbb{R}} I(\lambda u)=\frac{1}{2}\frac{(\int_{\mathcal{G}}u\mathrm{d} x)^{2} }{\int_{\mathcal{G}}\lvert u' \rvert^{2}\mathrm{d} x+\sum_{\mathsf{v} \in \mathsf{V}}\alpha_{\mathsf{v}} u(\mathsf{v})^{2}}\label{supeq},
\end{align*}
with equality if and only if 
	\[
		\lambda
		=
		\lambda_0
		:=
		\frac
		{(\int_{\mathcal{G}}u\mathrm{d} x) }
		{\int_{\mathcal{G}}\lvert u' \rvert^{2}\mathrm{d} x+\sum_{\mathsf{v} \in \mathsf{V}}\alpha_{\mathsf{v}} u(\mathsf{v})^{2}}.
	\] 
In particular, $\lambda_0^{-1} \upsilon$, where $\upsilon$ is the torsion function, i.e. the maximizer of $I$, is a maximizer of the P{\'o}lya quotient.
Since the P{\'o}lya quotient is invariant under multiplying the argument $u$ by non-zero scalars, any multiple of $\upsilon$ maximizes the P{\'o}lya quotient.
Conversely, let $u$ maximize the P{\'o}lya quotient.
By the above arguments, a scalar multiple of $u$ must be a maximizer of $I$, so $u$ must be a scalar multiple of the torsion function $\upsilon$.

We have seen that the functional $I$ is uniquely maximized by the torsion function $\upsilon$, and that the maximal value of $I$ coincides with the maximum of the P\'olya quotient.
It remains to be seen that the maximal value of $I$ is indeed the torsional rigidity.
This is verified by direct calculation, using $- v'' \equiv 1$ in $\mathcal{G}$, integration by parts, and the $\delta$-vertex conditions.
\end{proof}

\section{Including Dirichlet vertex conditions}
\label{sec:Dirichlet}

Previously, the torsional rigidity on metric graphs had been defined in the presence of vertices with Dirichlet conditions~\cite{pluem}.
Indeed, we can combine our setting of $\delta$-vertex with the presence of Dirichlet vertices.
For this, let $\mathsf{V}_D \subset \mathsf{V}$.
We define the Sobolev space of functions vanishing at $\mathsf{V}_D$
\begin{equation*}
	H^1_0(\mathcal{G}, \mathsf{V}_D)
	=
	\left\lbrace  u \in H^1(\mathcal{G})
	\
	\mid
	\
	u(\mathsf{v}) = 0 \text{ for all } \mathsf{v} \in \mathsf{V}_D
	\right\rbrace.
\end{equation*}

We then analogously define for $\beta \in [0, \infty)^{\mathsf{V} \backslash \mathsf{V}_D}$ the quadratic form $h_{\beta} = h_{\mathcal{G}, \beta, \mathsf{V}_D}$ with domain $H^1_0(\mathcal{G}, \mathsf{V}_D)$ as
\begin{equation*}
	h_{\beta}(u)
	=
	\int_{\mathcal{G}}\lvert u'(x) \rvert^{2}\mathrm{d} x+\sum_{\mathsf{v}  \in \mathsf{V} \backslash \mathsf{V}_D} \beta_{\mathsf{v} } \lvert u(\mathsf{v} ) \rvert^{2},
\end{equation*}
which again gives rise to a self-adjoint operator $\Delta_{\mathcal{G}, \beta, \mathsf{V}_D}$ in $L^2(\mathcal{G})$, the Laplacian with Dirichlet boundary conditions at $\mathsf{V}_D$ and $\delta$-vertex conditions at the other vertices.
Its domain consists of functions which are edgewise in the $H^2$ Sobolev space and satisfy the following vertex conditions:
\begin{equation}
\label{vertexcond}
\begin{cases}
	\text{$u$ is continuous } \text{ at every $\mathsf{v} \in \mathsf{V}$},
	\\
	u(\mathsf{v}) = 0 \text{ for all $\mathsf{v} \in \mathsf{V}_D$},
	\\
	\sum_{\mathsf{e} \in \mathsf{E}_{\mathsf{v} }}\frac{\partial u_{\mathsf{e}} }{\partial n}(\mathsf{v} )+\beta_{\mathsf{v} } u(\mathsf{v} )=0
	\text{ for all $\mathsf{v}  \in \mathsf{V} \backslash \mathsf{V}_D$}.
\end{cases}
\end{equation}
\begin{remark}[On connectedness]
	In the presence of vertices with Dirichlet boundary conditions, there is a slight issue with connectedness.
	Whenever we say that $\mathcal{G}$ with Dirichlet vertex set $\mathsf{V}_D \subset \mathsf{V}$ is \emph{connected} we mean that for all $x,y \in \mathcal{G}$ there is a path in $\mathcal{G} \backslash \mathsf{V}_D$ joining $x$ and $y$.
	This way, we ensure that Dirichlet vertices will not split $\mathcal{G}$ into spectrally independent components.   
\end{remark}

By positivity of the spectrum and the fact that $\Delta_{\mathcal{G}, \beta, \mathsf{V}_D}$ in $L^2(\mathcal{G})$ generates a positive semigroup, one also has the analogue of Proposition~\ref{postor}, cf.~\cite{pluem}.
\begin{proposition}\label{postor_dir}
Let $\mathcal{G}$ be a connected metric graph, $\emptyset \neq \mathsf{V}_D \subset \mathsf{V}$, and $\beta \in [0, \infty)^{ \mathsf{V} \backslash \mathsf{V}_D }$.
Then the equation
\begin{eqnarray}\label{torrob_dir}
	-\Delta_{\mathcal{G}, \beta, \mathsf{V}_D} u =1
\end{eqnarray}
has a unique strictly positive solution in $H^1(\mathcal{G})$.
\end{proposition}

\begin{definition}\label{torfunc_dir}
Let $\mathcal{G}$ be a connected metric graph, $\emptyset \neq \mathsf{V}_D \subset \mathsf{V}$, and $\beta \in [0, \infty)^{ \mathsf{V} \backslash \mathsf{V}_D}$.
The \emph{torsion function} $\upsilon$ of $\mathcal{G}$ is the unique solution $\upsilon$ of equation~\eqref{torrob_dir}.
The \emph{torsional rigidity} of $\mathcal{G}$ with Dirichlet conditions at $\mathsf{V}_D$ and strength $\beta$ is
\begin{equation}
	T(\mathcal{G}, \beta, \mathsf{V}_D)=\int_{\mathcal{G}} \upsilon(x) \mathrm{d} x.
\end{equation}
\end{definition}

\begin{theorem}\label{lemmavarchar_dir}
Let $\mathcal{G}$ be a connected metric graph, let $\emptyset \neq \mathsf{V}_D \subset \mathsf{V}$, and let $\beta \in [0, \infty)^{ \mathsf{V} \backslash \mathsf{V}_D}$.
Then
\begin{equation}\label{varchar_dir}
	T(\mathcal{G}, \beta, \mathsf{V}_D)
	=
	\sup_{u\in H^{1}_0(\mathcal{G}, \mathsf{V}_D)}
	\frac{(\int_{\mathcal{G}}\lvert u \rvert \mathrm{d} x)^{2} }{h_{\beta}(u)}
	=
	\max_{u\in H^{1}_0(\mathcal{G}, \mathsf{V}_D)}
	\frac{\lVert u \rVert^{2}_{L^{1}(\mathcal{G})} }{h_{\beta}(u)},
\end{equation}
if all strengths are non-negative and at least one of them is positive.
The torsion function is the unique maximizer in~\eqref{varchar}.
\end{theorem}
The proof of Theorem~\ref{lemmavarchar_dir} is completely analogous to the one of~\ref{thm:varchar}.
Intuitively, when the strengths at a set of vertices $\mathsf{W} \subset \mathsf{V}$ are sent to $\infty$, then Dirichlet conditions at $\mathsf{W}$ should emerge.
The next theorem makes this intuition precise on the level of torsional rigidity.

\begin{theorem}\label{conasp}
	Let $\mathcal{G}$ be a connected metric graph.
	Let $\mathsf{V}_D \subset \mathsf{V}$, and $\mathsf{W} \subset \mathsf{V} \backslash \mathsf{V}_D$.
	Let further $\beta_1 \in [0, \infty)^{ \mathsf{V} \backslash \{ W \cup \mathsf{V}_D \}}$, $\beta_2 \in (0, \infty)^{W}$.
	Then,
	\[
	\lim_{t \to \infty}
	T(\mathcal{G}, (\beta_1, t \beta_2), \mathsf{V}_D)
	=
	T(\mathcal{G}, \beta_1, \mathsf{V}_D \cup W).
	\]
	Furthermore, the respective torsion functions converge in $L^1(\mathcal{G})$.
\end{theorem}

Note that $\mathsf{V}_D$ is allowed to be the empty set in Theorem~\ref{conasp}.

\begin{proof}
	This is a direct consequence of norm resolvent convergence of Laplacians with $\delta$-vertex conditions to the Laplacian with Dirichlet conditions as the strengths tend to $\infty$, cf.~\cite[Theorem~3.5]{berk5}.
\end{proof}

\section{Reduction to a weighted discrete Laplacian}
\label{seclin}

In this section, we further investigate under which conditions the torsion function exists as a positive function.
Therefore, we drop the condition on strengths given in Proposition~\ref{postor}. 

\begin{remark}
	One could further relax our requirements and only ask that a solution $\upsilon$ to the torsion problem~\eqref{torrob} exists.
	Indeed, this would be equivalent to $0$ not being in the spectrum of $\Delta_{\mathcal{G}, \alpha}$.
	However, sign-changing torsion functions would probably undermine the concept of torsional rigidity (their integral) as a landscape function or a mechanical quantity since the latter could now become negative.
	We therefore believe that existence of a torsion function as a positive function is a reasonable minimal requirement.
\end{remark}

The purpose of this section is twofold.
First, we establish a connection between existence of the torsion function as a positive function and positivity of an associated \emph{weighted discrete Laplacian} where the weights depend on the topological and metric structure of the metric graph $\mathcal{G}$, and on the weights $\alpha$.
This illustrates that while non-negativity of all strengths as in Proposition~\ref{postor} might not be necessary to ensure existence of a positive torsion function, it is in general a hard problem, depending on positivity of a weighted discrete Laplacian.

Secondly, we use the translation to a discrete Laplace operator to further investigate the question of positivity of the torsion function  in our previous examples.
The computation of the torsion function on a metric graph can be reduced to an algebraic system of $\lvert \mathsf{V} \rvert$ many equations, cf. \cite[Proposition 3.1]{pluem}.

\begin{proposition} 
	\label{syslin}
	Let $\mathcal{G}$ be a connected metric graph with strengths $\alpha \in \mathbb{R}^{\mathsf{V}}$, and assume that its torsion function $\upsilon$ exists.
	Then, the restriction $g:=\upsilon \mid_{\mathsf{V}}\colon \mathsf{V} \to \mathbb{C}$ is the unique solution of the linear system
	\begin{equation}\label{system}
		\frac{2}{d^{l}_{\mathsf{v}}}\left( \sum_{\mathsf{e}=[\mathsf{v},\mathsf{w}]\in\mathsf{E}_{\mathsf{v}} }\left(\frac{g(\mathsf{v})-g(\mathsf{w})}{\ell_{\mathsf{e}}} \right)+\alpha_{\mathsf{v}}g(\mathsf{v})\right) =1
	\end{equation}
	where $d_{\mathsf{v}}^{\ell}=\sum_{\mathsf{e} \in\mathsf{E}_{\mathsf{v}} } \ell_\mathsf{e}$. 
	Conversely, if~\eqref{system} has a unique solution  $(g(\mathsf{v}))_{\mathsf{v}\in \mathsf{V} }$, then there exists a unique torsion function $\upsilon \in H^1(\mathcal{G})$ such that $g(\mathsf{v})=\upsilon(\mathsf{v})$ for all $\mathsf{v} \in \mathsf{V}$.
\end{proposition} 
	
\begin{proof}
		Since $-\upsilon_{\mathsf{e}}''=1$ on each edge, we get for all $x_{\mathsf{e}}\in [0,\ell_{\mathsf{e}}]$
		\begin{equation*}
			\upsilon_{\mathsf{e}}(x_{\mathsf{e}})=\frac{(\ell_{\mathsf{e}}-x_{\mathsf{e}})x_{\mathsf{e}}}{2}+\frac{\upsilon_{\mathsf{e}}(\ell_{\mathsf{e}})-\upsilon_{\mathsf{e}}(0)}{\ell_{\mathsf{e}}}x_{\mathsf{e}}+\upsilon_{\mathsf{e}}(0).
		\end{equation*}
		In particular,
		\begin{equation*}
			\upsilon^{'}_{e}(x_{\mathsf{e}})=\frac{\ell_{\mathsf{e}}-2x_{\mathsf{e}}}{2}+\frac{\upsilon_{\mathsf{e}}(\ell_{\mathsf{e}})-\upsilon_{\mathsf{e}}(0)}{\ell_{\mathsf{e}}},
		\end{equation*}
		for all $x_{\mathsf{e}}\in [0,\ell_{\mathsf{e}}]$. 		
		Implementing $\delta$-vertex conditions leads to
		\begin{eqnarray*}
			-\alpha_{\mathsf{v}}g(v)&=&\sum_{\mathsf{e}=[\mathsf{v},\mathsf{w}]\in\mathsf{E}_{\mathsf{v}} }-\frac{\ell_\mathsf{e}}{2}+\frac{g(\mathsf{v})-g(\mathsf{w})}{\ell_\mathsf{e}}\\
			&=&-\frac{d^{l}_{\mathsf{v}}}{2}+\sum_{\mathsf{e}=[\mathsf{v},\mathsf{w}]\in\mathsf{E}_{\mathsf{v}} }\frac{g(\mathsf{v})-g(\mathsf{w})}{\ell_\mathsf{e}},
		\end{eqnarray*}
		which yields~\eqref{system}.
		Conversely, for vertex values $(g(\mathsf{v}))_{\mathsf{v} \in \mathsf{V}}$, there is a unique function $\upsilon$ satisfying $-\upsilon_{\mathsf{e}}''=1$ on each $\mathsf{e} \in \mathsf{E}$, given by
		\begin{equation*}
			\upsilon_{\mathsf{e}}(x_{\mathsf{e}})=\frac{(\mathsf{w}-x_{\mathsf{e}})(x_{\mathsf{e}}-\mathsf{v})}{2}+\frac{g(\mathsf{w})-g(\mathsf{v})}{\mathsf{w}-\mathsf{v}}(x_{\mathsf{e}}-\mathsf{v})+g(\mathsf{v})
		\end{equation*}
		It satisfies $\delta$-vertex conditions at every vertex due to~\eqref{system}.
	Uniqueness follows from the one-to-one correspondence of $\upsilon$ and $g$.    
	\end{proof}

	The torsion function $\upsilon$ is strictly concave on each edge. Hence, positivity at all vertices is equivalent to positivity of $\upsilon$ itself. 
By Proposition~\ref{syslin}, positivity of the torsion function at vertices is equivalent to existence of a positive solution of the linear system~\eqref{system} of $\lvert \mathsf{V}\rvert$ many equations. 
	Let $L_{\mathcal{G}, \alpha}$ denote the matrix form of that linear system. 
	Note that the operator $L_{\mathcal{G}, \alpha}$ has the form of a \emph{weighted discrete Laplacian} with (not necessarily sign-definite) potential, cf.~\cite{KellerLW-21}. 
	Its entries are real and non-positive off-diagonal, hence, $(\mathrm{e}^{- L_{\mathcal{G}, \alpha} t})_{t \geq 0}$ generates a positive semigroup, cf.~\cite[Theorem~7.1]{batka}.
In particular, there exists $x\geq 0$ with $L_{\mathcal{G}, \alpha} x=\boldsymbol{1}$ if and only if the spectrum of $L_{\mathcal{G}, \alpha}$ is non-negative or -- alternatively -- its characteristic polynomial has no root in $(-\infty,0]$, see~\cite[Corollary 7.4]{batka}.
%Consequently, in order to show that all vertex values of the torsion function are positive, it suffices to ensure positivity of  the linear system~\eqref{syslin} has only positive eigenvalues.
Let us summarize these findings:

\begin{theorem}
	\label{thm:torsion_discrete_Laplacian}
Let	$\mathcal{G}$ be a connected metric graph with strengths $\alpha \in \mathbb{R}^{\mathsf{V}}$.
	The torsion function exists as a strictly positive function if and only if the infimum of the spectrum of the \emph{weighted discrete Laplacian} $L_{\mathcal{G}, \alpha}$ on $\ell^2(\mathsf{V})$, given in~\eqref{system} is strictly positive.
\end{theorem}

Before proceeding, let us discuss some examples.

\begin{example}[Interval graph]
	\label{exa:interval_graph}
	Consider an interval graph $\mathcal{J}$ of length $a > 0$ with strengths $\alpha_0, \alpha_1 \in \mathbb{R}$ at the end points, cf. Figure~\ref{fig:examples}.
	Identifying $\mathcal{J}$ with $[0,a]$, the torsion problem~\eqref{torrob} becomes
	\[
	\begin{cases}
		- \upsilon''(x) =1, &x \in (0,a),\\
		-\upsilon'(0)+\alpha_{0} \upsilon(0)=0, \\
		\upsilon'(a)+\alpha_{1} \upsilon(a)=0.
	\end{cases}
	\]
	The corresponding discrete Laplacian $L_{\mathcal{J}, \{ \alpha_0, \alpha_1\}}$ is then
	\[
	L_{\mathcal{J}, \{ \alpha_0, \alpha_1\}}
	=
	\begin{pmatrix}
		\frac{2}{a} \left( \frac{1}{a} + \alpha_0 \right)
		&
		- \frac{2}{a^2}\\
		- \frac{2}{a^2}
		&
		\frac{2}{a} \left( \frac{1}{a} + \alpha_1 \right)	
	\end{pmatrix}
	\]
	with characteristic polynomial
	\[
	\chi_{L_{\mathcal{J}, \{ \alpha_0, \alpha_1\}}}(x)=x^2-x\left(\frac{2(\alpha_{0}+\alpha_{1})}{a}+\frac{4}{a^2}\right)+ \frac{4}{a^2}\left(\alpha_{0}\alpha_{1}+ \frac{(\alpha_{0}+\alpha_{1})}{a}\right) .
	\]

	A solution (not necessarily sign-definite) to the torsion problem exists if and only if both roots are not zero.
	The torsion exists as \emph{positive function} if and only if both roots are non-negative, that is $\alpha_{0}+\alpha_{1} > \max \{-\frac{2}{a},-a\alpha_{0}\alpha_{1}\}$.
	In particular, this allows for configurations where one strength is negative, but the torsion function positive e.g. $\alpha_0 = -1$, $\alpha_1 = 2$, and $a = 1/4$.
	
	We finally note for later use that in this situation, the torsion function can be explicitly calculated.
%	 to be
%	\[
%	\upsilon(x)=\frac{-x^{2}}{2}+\frac{\alpha_{0} a (a\alpha_{1}+2)}{2(a\alpha_{0}\alpha_{1}+\alpha_{0}+\alpha_{1})}x+\frac{a(a \alpha_{1}+2)}{2(a\alpha_{0}\alpha_{1}+\alpha_{0}+\alpha_{1})}.
%	\]
	In particular, in the special case where $\alpha_0 > 0$ and $\alpha_1 = 0$, we find
	\[
	T(\mathcal{J}, \{\alpha_0, 0\})
	=
	\frac{\lvert \mathcal{J} \rvert^{3}}{3}+\frac{\lvert \mathcal{J} \rvert^{2}}{\alpha_0}.
	\]
\end{example}

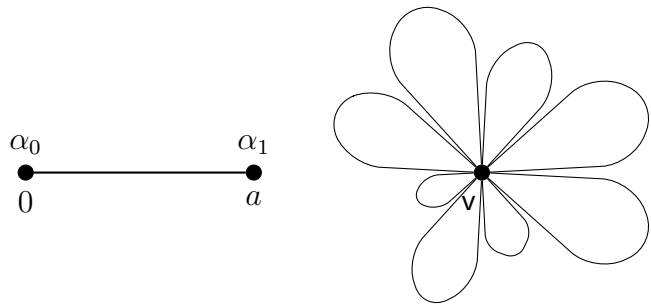
\begin{figure}
	\begin{tikzpicture}
		\begin{scope}
			\node[draw, circle, inner sep=2pt, fill, label=above:{$\alpha_0$}, label=below:{$0$}] (A) at (0, 0) {};
			\node[draw, circle, inner sep=2pt, fill,  label=above:{$\alpha_1$}, label=below:{$a$}] (B) at (3, 0) {};
			\draw[thick] (A) -- (B);
		\end{scope}

		\begin{scope}[xshift = 6cm]
			 \node[draw, circle, inner sep=2pt, fill] (A) at (0, 0) {};
			 \draw (-.17,-.4) node {$\mathsf{v}$};
			 \begin{scope}[rotate = 0]
			 	\draw[rounded corners = 12pt] (A) -- (3:2) -- (22.5:2.5) -- (42:2) -- (A);
			 \end{scope} 
			 \begin{scope}[rotate = 45]
			 	\draw[rounded corners = 10pt] (A) -- (3:1.5) -- (22.5:2) -- (42:1.5) -- (A);
			 \end{scope} 
			 \begin{scope}[rotate = 90]
			 	\draw[rounded corners = 12pt] (A) -- (3:2) -- (22.5:2.5) -- (42:2) -- (A);
			 \end{scope} 
			 \begin{scope}[rotate = 135]
			 	\draw[rounded corners = 12pt] (A) -- (3:1.75) -- (22.5:2.25) -- (42:1.75) -- (A);
			 \end{scope} 			 
			 \begin{scope}[rotate = 180]
			 	\draw[rounded corners = 5pt] (A) -- (3:.75) -- (22.5:1) -- (42:.75) -- (A);
			 \end{scope} 
			 \begin{scope}[rotate = 225]
			 	\draw[rounded corners = 10pt] (A) -- (3:1.5) -- (22.5:2) -- (42:1.5) -- (A);
			 \end{scope} 
			 \begin{scope}[rotate = 270]
			 	\draw[rounded corners = 5pt] (A) -- (3:1) -- (22.5:1.25) -- (42:1) -- (A);
			 \end{scope} 			 
			 \begin{scope}[rotate = 315]
			 	\draw[rounded corners = 12pt] (A) -- (3:2) -- (22.5:2.5) -- (42:2) -- (A);
			 \end{scope}

		\end{scope}
		
	\end{tikzpicture}
	
		\caption{The interval graph $\mathcal{J}$ from Example~\ref{exa:interval_graph} (left) and the flower graph $\mathcal{F}$ from Example~\ref{exa:flower} (right).}
	\label{fig:examples}
		
\end{figure}

\begin{example}[Flower graph]\label{exa:flower}
 	Let $\mathcal{F}$ be the flower graph with edge set $\mathsf{E}$ with a $\delta$-vertex condition of strength $\alpha$ at the only vertex $\mathsf{v}$, cf. Figure~\ref{fig:examples}. 
	The discrete Laplacian $L_{\mathcal{F}, \alpha}$ has the only entry
	\[
	L_{\mathcal{F}, \alpha}
	=
	\frac{\alpha}{\lvert \mathcal{F} \rvert}
	\]
	A solution to the torsion problem~\eqref{torrob} exists if and only if $\alpha \neq 0$, and the torsion function is positive if and only if $\alpha>0$.

	Let us also calculate the torsional rigidity for later use in the case $\alpha > 0$. 
	On each edge $\mathsf{e}$ the torsion function is
	\[
	\upsilon_{\mathsf{e}}(x_{\mathsf{e}})
	=
	\frac{-x_{\mathsf{e}}^{2}}{2}
	+
	\frac{\ell_{\mathsf{e}}}{2}x_{\mathsf{e}}+\frac{\lvert \mathcal{F} \rvert}{\alpha}
	\]
	which leads to the torsional rigidity
	\[
	T(\mathcal{F}, \alpha)=\sum_{\mathsf{e} \in \mathsf{E}}\frac{\ell_{\mathsf{e}}^{3}}{12}+\frac{\lvert \mathcal{F} \rvert^{2}}{\alpha}.
	\]
In particular, if $\mathcal{F}$ is an equilateral flower graph, then
\begin{equation*}
T(\mathcal{F}, \alpha)=\frac{\lvert \mathcal{F} \rvert^{3}}{12\lvert \mathsf{E} \rvert^{2}}+\frac{\lvert \mathcal{F} \rvert^{2}}{\alpha}.
\end{equation*}
\end{example}

\section{Hadamard-type formulas}
	\label{sec:hadamard}

In this section we study the behavior of the torsion function under perturbations of edge lengths and strengths.
Note that throughout this section we do not require on positivity of the torsion function but merely the existence of a solution of the torsion problem~\eqref{torrob}.

Fixing one edge $\mathsf{e}_{0}\in E$, for $s>0$, consider the perturbed graph $\mathcal{G}_{s}$ with the same topology as $\mathcal{G}$ and whose edge lengths $(\ell_{s,\mathsf{e}})_{\mathsf{e} \in \mathsf{E}}$ are
  $$\ell_{s,\mathsf{e}}=\begin{cases}
	s, &\mathsf{e}=\mathsf{e}_{0},\\
	\ell_{\mathsf{e}}, &\mathsf{e}\neq \mathsf{e}_{0}.
\end{cases}$$
 We obtain the original graph $\mathcal{G}$ for $s=\ell_{\mathsf{e}_{0}}$. 
 We denote the torsion function on $\mathcal{G}_{s}$ by $\upsilon_{s}={(\upsilon_{s,\mathsf{e}})}_{\mathsf{e}\in \mathsf{E}}$, implicitly dropping the dependence on the strengths $\alpha$.

\begin{theorem}\label{changeoflength}
	Let $\mathcal{G}$ be a connected metric graph and let $\alpha \in \mathbb{R}^{\mathsf{V}}$ such that the torsion exists.
	Then there exists a neighborhood $U\subset(0,\infty)$ of $\ell_{\mathsf{e}_{0}}$ such that the map $U \ni s \mapsto T(\mathcal{G}_{s}, \alpha)$ is continuously differentiable and its derivative at $s=\ell_{\mathsf{e}_{0}}$ is
	\begin{equation}\label{derivativeexpression}
		\frac{d}{ds}\bigg \vert_{s=\ell_{\mathsf{e}_{0}}}T(\mathcal{G}_{s}, \alpha)=2 \upsilon_{s,\mathsf{e}_{0}}(x_0)+\upsilon_{s,\mathsf{e}_{0}}'(x_0)^{2},
	\end{equation}
		where $x_0$ is any element of $\mathsf{e}_{0}$. In particular, the right-hand-side in~\eqref{derivativeexpression} does not depend on the particular choice of $x_0\in \mathsf{e}_{0}$. 
		Furthermore if  $\alpha \in [0,\infty)^{\mathsf{V}}$ and not identically zero, then $U=(0,\infty)$.
\end{theorem}
\begin{proof}
	The existence of a neighbourhood $U$ in which the torsion problem has a solution and differentiability follow from the reduction to a discrete Laplacian in Theorem~\ref{thm:torsion_discrete_Laplacian} and the implicit function theorem, using that all entries of $L_{\mathcal{G}, \alpha}$ depend smoothly on edge lengths.
	Using that the torsional rigidity $\upsilon$ satisfies $1 \equiv - \upsilon_s''$, integration by parts on each edge, and the $\delta$-vertex conditions, we have 
	\begin{equation*}
		T(\mathcal{G}_{s}, \alpha)=\int_{\mathcal{G}_{s}}\upsilon_{s}(x)\mathrm{d} x
		= 
		\int_{\mathcal{G}_{s}}\upsilon_{s}'(x)^{2}\mathrm{d} x+\sum_{\mathsf{v} \in \mathsf{V}}\alpha_{\mathsf{v}}\upsilon_{s}(\mathsf{v})^{2}.
	\end{equation*}
	We differentiate three terms in the previous identity.
	First we have for any $x_0 \in \mathsf{e}_0$
	\begin{equation*}\label{dertor0}
		\frac{d}{ds}\int_{\mathcal{G}_{s}}\upsilon_{s}(x)\mathrm{d} x
		=
		\upsilon_{s,\mathsf{e}_{0}}(s)+\int_{\mathcal{G}_{s}}\frac{d}{ds}\upsilon_{s}(x)\mathrm{d} x.
	\end{equation*}
	Using again integration by parts, $1 \equiv - \upsilon_s'' $, and the $\delta$-vertex conditions
	\begin{align*}
		&\frac{d}{ds}\int_{\mathcal{G}_{s}}\upsilon_{s}'(x)^{2}
		\mathrm{d} x
		=
		\upsilon_{s,\mathsf{e}_{0}}'(s)^{2}
		+
		2\sum_{\mathsf{e} \in \mathsf{E}}\int_{0}^{\ell_{\mathsf{e}}}\upsilon_{s,\mathsf{e}}'(x_{\mathsf{e}})\frac{d}{ds}\upsilon_{s,\mathsf{e}}'(x_{\mathsf{e}})\mathrm{d} x_{\mathsf{e}}\\
		&=
		\upsilon_{s,\mathsf{e}_{0}}'(s)^{2}
		-
		2\sum_{\mathsf{e}\in \mathsf{E}}\int_{0}^{\ell_{\mathsf{e}}}\upsilon_{s,\mathsf{e}}''(x_{\mathsf{e}})\frac{d}{ds}\upsilon_{s,\mathsf{e}}(x_{\mathsf{e}})\mathrm{d} x_{\mathsf{e}}
		\\
		& \qquad
		 + 2 \sum_{\mathsf{e} \in \mathsf{E}}
		 	\left[
			\upsilon_{s,\mathsf{e}}'(\ell_{s,\mathsf{e}})\frac{d}{ds}\upsilon_{s,\mathsf{e}}(\ell_{s,\mathsf{e}})-\upsilon_{s,\mathsf{e}}'(0)\frac{d}{ds}\upsilon_{s,\mathsf{e}}(0)
			\right]
		\\
		&=
		\upsilon_{s,\mathsf{e}_{0}}'(s)^{2}
		+
		2\int_{\mathcal{G}_{s}}\frac{d}{ds}\upsilon_{s}(x)\mathrm{d} x
		+ 2 \sum_{\mathsf{v} \in \mathsf{V}}\frac{d}{ds}\upsilon_{s}(\mathsf{v})\sum_{\mathsf{e} \in\mathsf{E}_{\mathsf{v}} }\frac{\partial \upsilon_{s,\mathsf{e}}(\mathsf{v})}{\partial n}
		\\
		&=
		\upsilon_{s,\mathsf{e}_{0}}'(s)^{2}
		+
		2\int_{\mathcal{G}_{s}}\frac{d}{ds}\upsilon_{s}(x)\mathrm{d} x
		- 
		2 \sum_{\mathsf{v}\in \mathsf{V}}\alpha_{\mathsf{v}}\upsilon_{s}(\mathsf{v})\frac{d}{ds}\upsilon_{s}(\mathsf{v}).
	\end{align*}
%	whence
%\begin{equation*}
%	\frac{d}{ds}\int_{\mathcal{G}_{s}} \upsilon_{s}'(x)^{2}\mathrm{d} x
%	=
%	2\int_{\mathcal{G}_{s}}\frac{d}{ds} 
%	\upsilon_{s}(x) 
%	\mathrm{d} x-2\sum_{\mathsf{v} \in \mathsf{V}}\alpha_{s,\mathsf{v}}\upsilon_{s}(\mathsf{v})\frac{d}{ds}\upsilon_{s}(\mathsf{v}).
%\end{equation*}
Finally, for the third term, we have
\begin{equation*}
\frac{d}{ds}\left( \sum_{\mathsf{v} \in \mathsf{V}}\alpha_{\mathsf{v}}\upsilon_{s}(\mathsf{v})^{2}\right)= 2\sum_{\mathsf{v} \in \mathsf{V}} \alpha_{\mathsf{v}}\upsilon_{s}(\mathsf{v}) \frac{d}{ds}\upsilon_{s}(\mathsf{v}).
\end{equation*}
Appropriate summation of the three expressions will cancel the summations and the integrals of $\mathcal{G}_s$, hence
\begin{align*}
	\frac{d}{ds}T(\mathcal{G}_{s}, \alpha)	
	&=
	2 \upsilon_{s,\mathsf{e}_{0}}(s)+\upsilon_{s,\mathsf{e}_{0}}'(s)^{2}.\end{align*}
	Now, since $- \upsilon'' = 1$, the expression $2 \upsilon_{s,\mathsf{e}_{0}}(x_0)+\upsilon_{s,\mathsf{e}_{0}}'(x_0)^{2}$ is independent of the position of $x_0$ on the edge $\mathsf{e}_{0}$.
\end{proof}
Note that in the case of Dirichlet conditions, an analogous Hadamard-type formula is known, see \cite[Proposition 3.11]{pluem}.

Generally, the torsional rigidity is a continuously differentiable function of the  edge lengths. 
Denoting by $\mathcal{G}_{\ell}$ the graph with side lengths $\ell \in (0, \infty)^{\mathsf{E}}$ and by $\upsilon_\ell$ its torsion function for a fixed $\alpha \in [0, \infty)^{\mathsf{V}}$ which is not identically zero, we have
\begin{corollary}
	\label{cor:derivative_lengths}
	The map $(0,\infty)^{\mathsf{E}} \ni \ell \mapsto T(\mathcal{G}_{\ell},\alpha)$ is continuously differentiable with differential 
	\[
	\left[
	2 \upsilon_{\ell,\mathsf{e}_1}(x_{\mathsf{e}_1})+\upsilon_{\ell,\mathsf{e}_{0}}'(x_{\mathsf{e}_1})^{2},
	\cdots,
	2 \upsilon_{\ell,\mathsf{e}_{\lvert \mathsf{E} \rvert}}(x_{\mathsf{e}_{\lvert \mathsf{E} \rvert}})+\upsilon_{\ell,\mathsf{e}_{\lvert \mathsf{E} \rvert}}'(x_{\mathsf{e}_{\lvert \mathsf{E} \rvert}})^{2}
	\right]
	\]
	where $x_{\mathsf{e}_1}, \dots, x_{\mathsf{e}_{\lvert \mathsf{E} \rvert}}$ are arbitrary points such that $x_{\mathsf{e}} \in \mathsf{e}$ for all $\mathsf{e} \in \mathsf{E}$.
\end{corollary}
	
\begin{proof}
	This is a straightforward consequence of Theorem~\ref{changeoflength}.
\end{proof}

Similarly, we can study the behavior of torsional rigidity under perturbation of the strengths of vertices. We fix one vertex $\mathsf{v}_{0}\in \mathsf{V}$. For $s>0$, we consider the same graph $\mathcal{G}$ of the same topology, and the Laplace operator $\Delta_{\mathcal{G},s}$ on $\mathcal{G}$ with $\delta$-vertex conditions with strengths 
$$\alpha_{s,\mathsf{v}}=\begin{cases}
	s, &\mathsf{v}=\mathsf{v}_{0},\\
	\alpha_{\mathsf{v}}, & \mathsf{v} \neq \mathsf{v}_{0}.
\end{cases}$$
 For $s >0$, the torsion on $\mathcal{G}$ will be denoted by $\upsilon_{s}=(\upsilon_{s,\mathsf{e}})_{\mathsf{e} \in \mathsf{E}}$.
 
\begin{theorem}\label{strengtder}
		Let $\mathcal{G}$ be a connected metric graph and let $\alpha \in \mathbb{R}^{\mathsf{V}}$ such that the torsion exists.
	Then there exists a neighborhood $U\subset \mathbb{R}$ of $\alpha_{{\mathsf{v}_{0}}} $ such that the map  $U\ni s  \mapsto T(\mathcal{G},\alpha_{s,\mathsf{v}})$ is continuously differentiable and its derivative at $s=\alpha_{{\mathsf{v}_{0}}}$ is
\begin{equation}
	\frac{d}{ds}\bigg \vert_{s=\alpha_{\mathsf{v}_{0}}}T(\mathcal{G},\alpha_{s,\mathsf{v}})= -\upsilon_{s}(\mathsf{v}_{0})^{2}.
\end{equation}
	Furthermore if  $\alpha \in [0,\infty)^{\mathsf{V}}$ and $\alpha_{{\mathsf{v}_{0}}}>0$, then $U$ can be chosen as $(0,\infty)$.
\end{theorem}

\begin{proof}
The proof is based on analogous calculations as in Theorem~\ref{changeoflength}.
\end{proof}

Analogously to Corollary~\ref{cor:derivative_lengths}, the torsional rigidity is a differentiable function of the strengths. 
Denote by $\upsilon_\alpha$ the torsion function of $\Delta_{\mathcal{G}, \alpha}$.
\begin{corollary}
	\label{cor:derivative_strengths}
	The map $(0,\infty)^{\mathsf{V}} \ni \alpha \mapsto T(\mathcal{G},\alpha)$ is continuously differentiable with differential $[-\upsilon_{\alpha}^{2}(\mathsf{v}_1),\cdots,-\upsilon_{\alpha}^{2}(\mathsf{v}_{\lvert \mathsf{V} \rvert})]$.
	\begin{proof}
This is a straightforward consequence of Theorem~\ref{strengtder}.
	\end{proof}
\end{corollary}

\section{Surgery principles}\label{surpri}

From now on, we stick to the setting where all strengths are non-negative and at least one strictly positive. 

Surgery principles on a metric graph investigate how modifications to the shape of a graph will affect quantities such as the spectrum or the torsional rigidity.
These modifications include lengthening or shortening of edges, connecting or cutting vertices, attaching a new graph to the existing one at one of its vertices or transplanting pendant edges to another parts of the graph in a process called \emph{unfolding}. 
In this section, we see that in the presence of vertex strengths, the theory of surgery  becomes richer than in graphs with only Dirichlet- and Kirchhoff-Neumann vertex conditions.

\begin{theorem}[Lengthening edges]
	\label{lengthedge}
	Let $\mathcal{G}$ be a connected metric graph and let $\alpha \in [0, \infty)^{\mathsf{V}}$ be not identically zero.
	Let $\tilde{\mathcal{G}} $ be the metric graph obtained from $\mathcal{G}$ by increasing the length of one edge, i.e. there exists $\mathsf{e}_{0}$ such that $\tilde{\ell_{\mathsf{e}}}=\ell_{\mathsf{e}}$ for $\mathsf{e} \neq \mathsf{e}_{0}$ and $\tilde{\ell_{\mathsf{e}_{0}}}>\ell_{\mathsf{e}_{0}}$. 
	Then $T(\mathcal{G}, \alpha)< T(\tilde{\mathcal{G}}, \alpha).$
\begin{proof}
		This follows from the Hadamard-type formula, Theorem~\ref{changeoflength}.
\end{proof}
\end{theorem} 
%\begin{remark}
%Theorem~\ref{lengthedge} shows in particular that torsional rigidity does not only depend on the number of vertices or strengths but also on edge lengths. 
%\end{remark}

\begin{theorem}[Simultaneous scaling of edges and strengths]
Let $\mathcal{G}$ be a connected metric graph and let $\alpha \in [0, \infty)^{\mathsf{V}}$ be not identically zero.
	Let $t >0$ and let $\tilde{\mathcal{G}}$ be the metric graph obtained by scaling each edge in $\mathcal{G}$ by $t$, i.e. $\tilde{\ell_{\mathsf{e}}}=t \ell_{\mathsf{e}}$ for $\mathsf{e} \in \mathsf{E}$.
\\[1em]	
%	 and $\tilde{\alpha_{\mathsf{v}}}=\frac{1}{t}\alpha_\mathsf{v}$, for $\mathsf{v} \in \mathsf{V}$. 
	 Then $t^{-3}T(\tilde{\mathcal{G}}, \frac{1}{t}\alpha)=T(\mathcal{G}, \alpha)$.
\begin{proof}
	Let $h_{\alpha},\tilde{h}_{ \frac{1}{t}\alpha}$ be the quadratic forms associated with the Laplace operators $\Delta_{\mathcal{G}, \alpha}, \Delta_{\tilde{\mathcal{G}}, \frac{1}{t} \alpha}$, respectively. 
	For $f\in\operatorname{dom}(H)=H^{1}(\mathcal{G})$, and its restriction $f_{\mathsf{e}}$ to an arbitrary edge $\mathsf{e} \in \mathsf{E}$, parametrised as $[0,\ell_{\mathsf{e}}]$, define a function $\tilde{f}$ on $\tilde{\mathcal{G}}$ by 
	$$\tilde{f_{\mathsf{e}}}:=f_{\mathsf{e}}\left( \frac{x}{t}\right) ,$$ 
	$x \in [0,t\ell_{\mathsf{e}}]$. 
	Then, the values of $\tilde{f_{\mathsf{e}}}$ and $f_{\mathsf{e}}$ at each vertex are the same. Furthermore $f \in H^{1}(\mathcal{G})$ if and only if $\tilde{f} \in H^{1}(\tilde{\mathcal{G}})$, and we conclude
	$y=\frac{x}{t}$ 
	$$\tilde{h}_{ \frac{1}{t}\alpha}(\tilde{f})=\sum_{\mathsf{e} \in \mathsf{E}} \int_{0}^{t\ell_{\mathsf{e}}}\lvert \tilde{f_{\mathsf{e}}}'(x) \rvert^{2}\mathrm{d} x +\sum_{\mathsf{v} \in \mathsf{V}}\frac{\alpha_{\mathsf{v}}}{t}\lvert \tilde{f}(\mathsf{v}) \rvert^{2}=\frac{1}{t}h_{\alpha}(f).$$
Moreover $$\left( \int_{\tilde{\mathcal{G}}} \lvert \tilde{f}(x) \rvert \mathrm{d} x\right) ^{2}=t^{2}\left(\int_{\mathcal{G}}\lvert f(x) \rvert\mathrm{d} x\right)^{2}.$$
Altogether, we found $t^{-3}T(\tilde{\mathcal{G}},\frac{1}{t}\alpha)=T(\mathcal{G},\alpha)$.\qedhere
\end{proof}
\end{theorem}

\begin{theorem}[Gluing vertices] \label{gluver}
	Let $\mathcal{G}$ be a connected metric graph and let $\alpha \in [0, \infty)^{\mathsf{V}}$ be not identically zero.
	Let $\mathsf{v}_{1},\mathsf{v}_{2}\in \mathsf{V}$ with $\mathsf{v}_1 \neq \mathsf{v}_2$. Let $\alpha_{\mathsf{v}_{1}},\alpha_{\mathsf{v}_{2}}$ be the strengths at $\mathsf{v}_{1},\mathsf{v}_{2}$ respectively. 
	\\
	Denote by $\tilde{\mathcal{G}}$ the graph obtained from $\mathcal{G}$ by joining $\mathsf{v}_1$ and $\mathsf{v}_2$ to form one single vertex $\mathsf{v}_0$, and let 
	$\tilde	\alpha \in [0, \infty)^{\mathsf{V} \cup \{ \mathsf{v}_0 \} \backslash \{ \mathsf{v}_1, \mathsf{v}_2 \}}$ be defined by
	\[
	\tilde \alpha_{\mathsf{v}}
	=
	\begin{cases}
	\alpha_{\mathsf{v}_1} + \alpha_{\mathsf{v}_2}
	&
	{\mathsf{v}} = {\mathsf{v}_0},
	\\
	\alpha_{\mathsf{v}}
	&
	\text{else}.
	\end{cases}
	\]
	Then $T(\mathcal{G}, \alpha)\geq T(\tilde{\mathcal{G}}, \tilde \alpha).$
\begin{proof}
Let $h_{ \alpha},\tilde{h}_{\tilde \alpha}$ be the quadratic forms associated with the Laplace operators $\Delta_{\mathcal{G}, \alpha}, \Delta_{\tilde{\mathcal{G}}, \tilde \alpha}$, respectively. Then $H^{1}(\tilde{\mathcal{G}})\subset H^{1}(\mathcal{G}) $. Indeed, any $f\in H^{1}(\tilde{\mathcal{G}})$ satisfies $f(\mathsf{v}_0)=f(\mathsf{v}_1)=f(\mathsf{v}_2)$ and 
$$\tilde{h}_{\tilde \alpha}(f)- h_{\alpha}(f)=\tilde \alpha_{\mathsf{v}_{0}} \lvert f(\mathsf{v}_0)\rvert^{2}-\alpha_{\mathsf{v}_{1}} \lvert f(\mathsf{v}_1) \rvert^{2}-\alpha_{\mathsf{v}_{2}} \lvert f(\mathsf{v}_2) \rvert^{2}=0.$$
The result now follows from Theorem~\ref{varchar}.\qedhere
\end{proof}

\end{theorem}

Clearly, the inverse of the gluing vertices is \emph{cutting} through vertices. 
Let us formulate it here precisely for the sake of clarity: 

\begin{corollary}[Cutting vertices]\label{cuttingvertices}
Let $\mathcal{G}$ be a connected metric graph and let $\alpha \in [0, \infty)^{\mathsf{V}}$ be not identically zero.
Let $\mathsf{v}_{0}\in \mathsf{V}$ with  strength $\alpha_{\mathsf{v}_{0}}$. 
Let $\tilde{\mathcal{G}}$ be a graph obtained from $\mathcal{G}$ by cutting $\mathsf{v_0}$ into two vertices $\mathsf{v}_1$ and $\mathsf{v}_2$, and let 
$\tilde	\alpha \in [0, \infty)^{\mathsf{V} \cup \{ \mathsf{v}_1, \mathsf{v}_2 \} \backslash \{ \mathsf{v}_0 \} }$ be defined by
\[
\tilde \alpha_{\mathsf{v}}
=
\begin{cases}
	\alpha_{\mathsf{v}_1}  
	&
	{\mathsf{v}} = {\mathsf{v}_1},
	\\
	\alpha_{\mathsf{v}_2}
	 &
	{\mathsf{v}} = {\mathsf{v}_2},
	\\
	\alpha_{\mathsf{v}}
	&
	\text{else},
\end{cases}
\]
where $\alpha_{{\mathsf{v}_{0}}}=\alpha_{{\mathsf{v}_{1}}}+\alpha_{{\mathsf{v}_{2}}}$. 
If $\tilde{\mathcal{G}}$ is connected, then $T(\mathcal{G}, \alpha)\leq T(\tilde{\mathcal{G}}, \tilde \alpha).$
\end{corollary}

The following theorem states that $T(\mathcal{G})$ monotonically decreases with respect to strengths.
\begin{theorem}[Changing the strength of a vertex]
	\label{changingstrength}
	Let $\mathcal{G}$ be a connected metric graph, let $\mathsf{v}_0 \in \mathsf{V}$ and let $\alpha, \tilde \alpha \in [0, \infty)^{\mathsf{V}}$ be not identically zero  with
	\[
	\tilde \alpha_{\mathsf{v}_{0}} < \alpha_{\mathsf{v}_{0}}, 
	\quad
	\text{and}
	\quad
	\tilde \alpha_{\mathsf{v}_{0}} = \alpha_{\mathsf{v}_{0}}
	\quad
	\text{for $\mathsf{v} \neq \mathsf{v}_0$}.
	\]
	Then $T(\mathcal{G}, \alpha)< T(\mathcal{G}, \tilde \alpha).$
\end{theorem}

\begin{proof}
This follows from the Hadamard-type formula, Theorem~\ref{strengtder}, together with positivity	of the torsion function.
\end{proof}

\begin{definition}
	Let $\mathcal{G}$ be a connected metric graph.	
	Let $\mathsf{v}_{0} \in \mathsf{V}$, and let $\mathcal{G}'$ be another connected metric graph. 
	We say that a metric graph $\tilde{\mathcal{G}}$ is obtained by inserting $\mathcal{G}'$ into $\mathcal{G}$ at $\mathsf{v}_{0}$ if it is formed by removing $\mathsf{v}_{0}$ from $\mathcal{G}$ and attaching each of the edges $\mathsf{e} \in \mathsf{E}_{\mathsf{v}_{0}}$ to one of the vertices of $\mathcal{G}'$.
\end{definition}

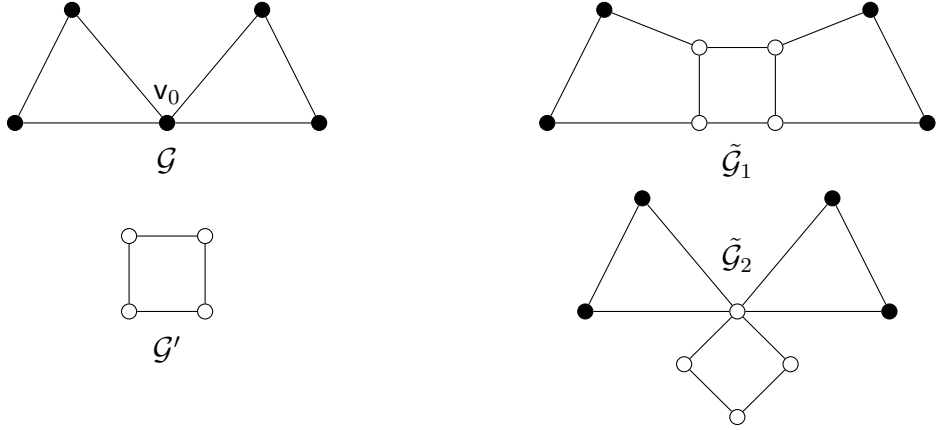
\begin{figure}
	\begin{tikzpicture}
		\begin{scope}
			\node[draw, circle, inner sep=2pt, fill, label=above:{$\mathsf{v}_0$}] (A) at (0, 0) {};
			\node[draw, circle, inner sep=2pt, fill] (B) at (-2,0) {};
			\node[draw, circle, inner sep=2pt, fill] (C) at (-1.25,1.5) {};
			\node[draw, circle, inner sep=2pt, fill] (D) at (2,0) {};
			\node[draw, circle, inner sep=2pt, fill] (E) at (1.25,1.5) {};

			\draw (A) -- (B) -- (C) -- (A) -- (D) -- (E) -- (A);
			
			\draw (0,-.5) node {$\mathcal{G}$};
		\end{scope}
		
		\begin{scope}[xshift = -.5cm, yshift = -2.5cm]
			\node[draw, circle, inner sep=2pt] (A) at (0,0) {};
			\node[draw, circle, inner sep=2pt] (B) at (1,0) {};
			\node[draw, circle, inner sep=2pt] (C) at (1,1) {};
			\node[draw, circle, inner sep=2pt] (D) at (0,1) {};
			\draw (A) -- (B) -- (C) -- (D) -- (A);
			\draw (.5,-.5) node {$\mathcal{G}'$};
			
		\end{scope}
		
		\begin{scope}[xshift = 7cm, yshift = 0cm ]
			\node[draw, circle, inner sep=2pt] (A) at (0, 0) {};
			\node[draw, circle, inner sep=2pt, fill] (B) at (-2,0) {};
			\node[draw, circle, inner sep=2pt, fill] (C) at (-1.25,1.5) {};
			\node[draw, circle, inner sep=2pt] (D) at (0,1) {};
			\node[draw, circle, inner sep=2pt] (E) at (1,0) {};
			\node[draw, circle, inner sep=2pt, fill] (F) at (3,0) {};
			\node[draw, circle, inner sep=2pt, fill] (G) at (2.25,1.5) {};		
			\node[draw, circle, inner sep=2pt] (H) at (1,1) {};		
			\draw (A) -- (B) -- (C) -- (D) -- (A);
			\draw (E) -- (F) -- (G) -- (H) -- (E);
			\draw (A) --(E);
			\draw (D) -- (H);
			
			\draw (.5,-.5) node {$\tilde{\mathcal{G}}_1$};
		\end{scope}
		
			\begin{scope}[xshift = 7.5cm, yshift = -2.5cm]
			\node[draw, circle, inner sep=2pt] (A) at (0, 0) {};				
			\node[draw, circle, inner sep=2pt, fill] (B) at (-2,0) {};
			\node[draw, circle, inner sep=2pt, fill] (C) at (-1.25,1.5) {};
			\node[draw, circle, inner sep=2pt, fill] (D) at (2,0) {};
			\node[draw, circle, inner sep=2pt, fill] (E) at (1.25,1.5) {};
			
			\node[draw, circle, inner sep=2pt] (F) at (.7, -.7) {};	
			\node[draw, circle, inner sep=2pt] (G) at (0,-1.4) {};	
			\node[draw, circle, inner sep=2pt] (H) at (-.7, -.7) {};

			\draw (A) -- (B) -- (C) -- (A) -- (D) -- (E) -- (A);
			\draw (A) -- (F) -- (G) -- (H) -- (A);
			
			\draw (0,.75) node {$\tilde{\mathcal{G}}_2$};
		\end{scope}

	\end{tikzpicture}
	
	\caption{The process of inserting $\mathcal{G}'$ into $\mathcal{G}$ at a vertex $\mathsf{v}_0 \in \mathcal{G}$ is not unique.
	Both $\tilde{\mathcal{G}_1}$ and $\tilde{\mathcal{G}_2}$ can be obtained.}
	\label{fig_inserting}
\end{figure}

We refer to Figure~\ref{fig_inserting} for an illustration. 
Note that the process of inserting a metric graph into another one at a vertex is not necessarily unique.

\begin{theorem}[Inserting a graph at a vertex]
	\label{insgrap}
		Let $\mathcal{G}$, $\mathcal{G}'$, be connected metric graphs with vertex set $\mathsf{V}$ and $\mathsf{V}'$, respectively. 
		Let $\alpha \in [0, \infty)^{\mathsf{V}}$ and $\alpha' \in [0, \infty)^{\mathsf{V}'}$  be not identically zero.
		Let $\mathsf{v_0} \in \mathsf{V}$, and let $\tilde{\mathcal{G}}$ be a metric graph formed by inserting $\mathcal{G}'$ into $\mathcal{G}$ at $\mathsf{v}_{0}$.
		Let $\tilde \alpha \in [0, \infty)^{\mathsf{V} \cup \mathsf{V}' \backslash \{ \mathsf{v}_0 \}}$ be defined via
		\[
		\tilde \alpha_{\mathsf{v}}
		=
		\begin{cases}
			\alpha_{\mathsf{v}}
			&
			\mathsf{v} \in \mathsf{V},\\
			\alpha'_{\mathsf{v}},
			&
			\mathsf{v} \in \mathsf{V}'.
		\end{cases}						
		\]
		If	
		\[
		\sum_{\mathsf{v}' \in \mathsf{V}'}\alpha_{\mathsf{v}'} \leq \alpha_{\mathsf{v}_{0}},
		\]
		then
		\[
		T(\mathcal{G}, \alpha) \leq T(\tilde{\mathcal{G}}, \tilde \alpha).
		\]
\end{theorem}
	\begin{proof}
		Let $f\in H^{1}(\mathcal{G})$. Extend $f$ to $\tilde{\mathcal{G}}$ constantly, i.e. $\tilde{f}:=f(\mathsf{v}_{0})$ on $\mathcal{G}'$. Then 
		\begin{eqnarray*}
			\tilde{h}_{\tilde \alpha}(\tilde{f}) &=& \int_{\tilde{\mathcal{G}}} \lvert \tilde{f}'(x) \rvert^{2}\mathrm{d} x+\sum_{w \in \mathsf{V}\setminus\{\mathsf{v}_{0}\}}\alpha_{\mathsf{v}} \lvert f(v) \rvert^{2}+ \lvert f(\mathsf{v}_{0}) \rvert^{2}\sum_{v' \in \mathsf{V}'}\alpha_{\mathsf{v}'}\\
			&  \leq & h_{ \alpha}(f)-\alpha_{\mathsf{v}_{0}} \lvert f(\mathsf{v}_{0}) \rvert^{2}+ \alpha_{\mathsf{v}_{0}} \lvert f(\mathsf{v}_{0}) \rvert^{2}
			=
			h_{ \alpha}(f).
		\end{eqnarray*}
		Furthermore,
		$$\left( \int_{\tilde{\mathcal{G}}} \lvert f(x) \rvert \mathrm{d} x\right) ^{2}=\left( \int_{\mathcal{G}} \lvert f(x) \rvert \mathrm{d} x + \lvert f(\mathsf{v}_{0})\rvert \lvert \mathcal{G}' \rvert \right)^{2} \geq \left( \int_{\mathcal{G}}\lvert f(x) \rvert\mathrm{d} x\right)^{2}.  $$
		This shows the claim.
	\end{proof}

\begin{example}
	\label{exa:inserting}
Let us show that the assumption $\sum_{\mathsf{v}' \in \mathsf{V}'}\alpha_{\mathsf{v}'} \leq \alpha_{\mathsf{v}_{0}}$ in Theorem~\ref{insgrap} cannot be dropped. 

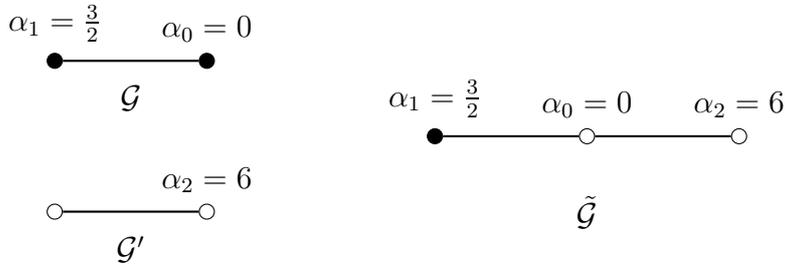
\begin{figure}[ht]
		\begin{tikzpicture}
		\begin{scope}[yshift = 1cm] 
			\node[draw, circle, inner sep=2pt, fill, label=above:{$\alpha_{1}=\frac{3}{2}$}] (A) at (0, 0) {};
			\node[draw, circle, inner sep=2pt, fill,  label=above:{$\alpha_{0}=0$}] (B) at (2, 0) {};
			\draw[thick] (A) -- (B);
			\draw (1,-.5) node {$\mathcal{G}$};
		\end{scope}

		\begin{scope}[yshift = -1cm] 
			\node[draw, circle, inner sep=2pt] (A) at (0, 0) {};
			\node[draw, circle, inner sep=2pt,  label=above:{$\alpha_{2}=6$}] (B) at (2, 0) {};
			\draw[thick] (A) -- (B);
			\draw (1,-.5) node {$\mathcal{G}'$};
		\end{scope}

		\begin{scope}[xshift = 3cm]
			\node[draw, circle, inner sep=2pt, label=above:{$\alpha_{2}=6$}] (D) at (6, 0) {};
			\node[draw, circle, inner sep=2pt, label=above:{$\alpha_{0}=0$}] (E) at (4, 0) {};
			\node[draw, circle, inner sep=2pt, fill, label=above:{$\alpha_{1}=\frac{3}{2}$}] (F) at (2, 0) {};
			\draw (4,-1) node {$\tilde{\mathcal{G}}$};
			\draw[thick] (D) -- (E);
			\draw[thick] (E) -- (F);
		\end{scope}

		\end{tikzpicture}
		
	\caption{Inserting an interval graph $\mathcal{G}'$ in an interval graph $\mathcal{G}$ to obtain the graph $\tilde{\mathcal{G}}$ as in Example~\ref{exa:inserting}}
	\label{fig:inserting}
\end{figure}
Consider an interval graph $\mathcal{G}$ of length $1$ with strengths $\frac{3}{2}$ and $0$ at its endpoints, see the left graph in Figure~\ref{fig:inserting}.
Its torsional rigidity is $1$. 
Now insert another interval graph $\mathcal{G}'$ of length $1$ to this graph at the vertex with strength $0$, and set the strength to be $6$ at the new vertex in the resulting graph $\tilde{\mathcal{G}}$, see the graph on the left in Figure~\ref{fig:inserting}.  
Its torsional rigidity is $\frac{23}{33}<1$, 
showing that the the assertion of Theorem~\ref{insgrap} cannot hold in this case.
\end{example}

In a metric graph $\mathcal{G}$, we say that an edge $\mathsf{e}$ is \emph{pendant} if one of its vertices has degree one.

\begin{theorem}[Unfolding pendant edges]
	\label{thm:unfolding}
	Let $\mathcal{G}$ be a connected metric graph and let $\alpha \in [0, \infty)^{\mathsf{V}}$ be not identically zero.
	Let edges $\mathsf{e}_{1},\dots,\mathsf{e}_{r}$ be pendant and incident to the same vertex $\mathsf{v}_{0}$, and denote by $\mathsf{v}_{j}$ the end points of $e_{j}$ respectively. 
	Let $\tilde{\mathcal{G}}$ be the metric graph obtained by replacing $\mathsf{e}_{1},\dots,\mathsf{e}_{r}$ by a single pendant edge $\hat{\mathsf{e}}$ of length $\ell_{\hat{\mathsf{e}}} = \sum_{j=1}^r \ell_{\mathsf{e}_{j}}$, and denote the vertex at the end point of $\hat{\mathsf{e}}$ by $\hat{\mathsf{v}}$.
	Let $\tilde \alpha \in [0, \infty)^{\mathsf{V} \cup \{ \hat{\mathsf{v}} \} \backslash \{ \mathsf{v}_{1},\dots,\mathsf{v}_{r} \}}$ be defined by
	\[
		\tilde \alpha_{\mathsf{v}}
		=
		\alpha_{\mathsf{v}}
		\quad
		\text{if $\mathsf{v} \in \mathsf{V}$, and}
		\quad		
		\tilde	\alpha_{\hat{\mathsf{v}}} = 0.
	\]
	Then $T(\mathcal{G}, \alpha)\leq T(\tilde{\mathcal{G}}, \tilde \alpha)$.
\end{theorem}	

We refer to Figure~\ref{fig:unfolding} for an illustration of the unfolding process.

\begin{center}
	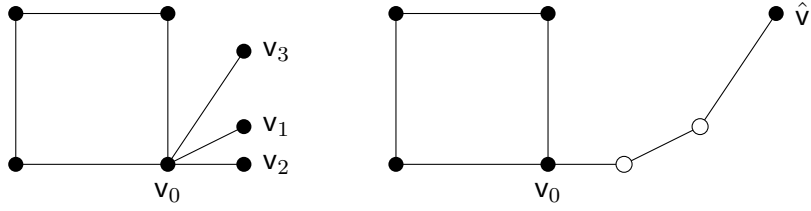
\begin{figure}
			\begin{tikzpicture}
			\begin{scope}
			% Define the coordinates of the square vertices
			\coordinate (A) at (0,0);
			\coordinate (B) at (2,0);
			\coordinate (C) at (2,2);
			\coordinate (D) at (0,2);
			% Draw the square
			\draw (A) -- (B) -- (C) -- (D) -- cycle;
			
			% Draw the extra edge from the bottom-right corner (B)
			\coordinate (E) at (3,.5); % Adjust the coordinates of E as necessary
			\coordinate (F) at (3,0);
			\coordinate (G) at (3,1.5);
			\draw (B) -- (E);
			\draw (B) -- (F);
			\draw (B) -- (G);

			% Add nodes to the vertices for clarity
			\node[fill, circle, inner sep=2pt] at (A) {};
			\node[fill, circle, inner sep=2pt, label=below:{$\mathsf{v}_0$}] at (B) {};
			\node[fill, circle, inner sep=2pt] at (C) {};
			\node[fill, circle, inner sep=2pt] at (D) {};
			\node[fill, circle, inner sep=2pt, label=right:{$\mathsf{v}_1$}] at (E) {};
			\node[fill, circle, inner sep=2pt, label=right:{$\mathsf{v}_2$}] at (F) {};
			\node[fill, circle, inner sep=2pt, label=right:{$\mathsf{v}_3$}] at (G) {};
		\end{scope}
		\begin{scope}[xshift = 5cm]
		
		% Define the coordinates of the square vertices
		\coordinate (A) at (0,0);
		\coordinate (B) at (2,0);
		\coordinate (C) at (2,2);
		\coordinate (D) at (0,2);
		% Draw the square
		\draw (A) -- (B) -- (C) -- (D) -- cycle;
		
		% Draw the extra edge from the bottom-right corner (B)
		\coordinate (E) at (3,0); % Adjust the coordinates of E as necessary
		\coordinate (F) at (4,.5);
		\coordinate (G) at (5,2);		
		\draw (B) -- (E) -- (F) -- (G);
	%	\draw (E) -- (F);
		
		% Add nodes to the vertices for clarity
		\node[fill, circle, inner sep=2pt] at (A) {};
		\node[fill, circle, inner sep=2pt, label=below:{$\mathsf{v}_0$}] at (B) {};
		\node[fill, circle, inner sep=2pt] at (C) {};
		\node[fill, circle, inner sep=2pt] at (D) {};
		\draw[black, fill = white] (E) circle (3pt);
		\draw[black, fill = white] (F) circle (3pt);
		\node[fill, circle, inner sep=2pt, label=right:{$\hat{\mathsf{v}}$}] at (G) {};
	\end{scope}
	\end{tikzpicture}

\caption{Unfolding pendant edges}
\label{fig:unfolding}
\end{figure}
\end{center}
	\begin{proof}
		By induction, we may assume $r=2$. 
 Let $f \in H^{1}(\mathcal{G})$. Define $\tilde{f}$ on $\tilde{\mathcal{G}}$ as equal to $f$ on all edges except $\hat{\mathsf{e}}$. 
 	Without loss of generality, let $f$ let take its maximum on $\mathsf{e}_1 \cup \mathsf{e}_2$ in a point $x_0 \in \mathsf{e}_1$.
	Note that $x_0$ could be identical to $\mathsf{v}_0$.
 We identifty the new edge $\hat{\mathsf{e}}$ with $[0,\ell_{ \mathsf{e}_{1}}+\ell_{ \mathsf{e}_{2}}]$ and define $\tilde{f}$ as
	\begin{align*}
\tilde{f}_{\hat{\mathsf{e}}}(x)=\begin{cases}
	f_{{{\mathsf{e}}}_{1}}(x), & x\in [0,x_{0}] \\
	f_{ {\mathsf{e}}_{1}}(x_0), & x \in [x_{0},x_{0}+\ell_{ \mathsf{e}_{2}}]\\
	f_{{ \mathsf{e}}_{1}}(x-\ell_{ \mathsf{e}_{2}}) & x \in [x_{0}+\ell_{ \mathsf{e}_{2}},\ell_{ \mathsf{e}_{1}}+\ell_{ \mathsf{e}_{2}}],
\end{cases}
	\end{align*}
		where $x_{0}\in {\mathsf{e}}_{1}$ is such that $f \mid_{{\mathsf{e}}_{1}\cup {\mathsf{e}}_{2}}$ takes its maximum at $x_{0}$. 
		Then the function $\tilde{f}$ lies in the domain of the quadratic form $h_{\tilde \alpha}$ associated with $\Delta_{\tilde{\mathcal{G}}, \tilde \alpha}$. Then
		\begin{align*}
				h_{	\tilde \alpha}(\tilde{f})
				&=h_{\alpha}(f)-\int_{{\mathsf{e}}_{1}\cup {\mathsf{e}}_{2}} \lvert f'(x) \rvert^{2}\mathrm{d} x-\sum_{j=1}^{2} \alpha_{j} \lvert f(\mathsf{v}_{j}) \rvert^{2}+\int_{\hat{\mathsf{e}}} \lvert \tilde f'(x) \rvert^{2}\mathrm{d} x
				\leq h_{\alpha}(f)
			\end{align*}
		since $\int_{{\mathsf{e}}_{1}\cup {\mathsf{e}}_{2}} \lvert f'(x) \rvert^{2}\mathrm{d} x\geq \int_{\hat{\mathsf{e}}} \lvert \tilde f'(x) \rvert^{2}\mathrm{d} x$ and $\alpha_{j} \geq 0$. Moreover, 
		\begin{equation*}
				\left( \int_{\tilde{\mathcal{G}}}\lvert \tilde{f} \rvert\mathrm{d} x \right)^{2} =\left(\int_{\mathcal{G}}\lvert f \rvert\mathrm{d} x-\int_{{\mathsf{e}}_{1}\cup {\mathsf{e}}_{2}}\lvert f \rvert\mathrm{d} x+\int_{\hat{\mathsf{e}}}\lvert \tilde{f} \rvert\mathrm{d} x \right)^{2} \geq \left(\int_{\mathcal{G}} \lvert f \rvert \mathrm{d} x \right)^{2}, 
			\end{equation*}
		where the last inequality follows from
		\begin{align*}
				\int_{\hat{\mathsf{e}}}\lvert \tilde{f} \rvert\mathrm{d} x 
				&=
				\int_{0}^{x_{0}} {\lvert f_{ {\mathsf{e}}_{1}} \rvert}\mathrm{d} x+\ell_{{\mathsf{e}}_{2}} \lvert f_{ {\mathsf{e}}_{1}}(x_{0})\rvert +\int_{x_{0+\ell_{{\mathsf{e}}_{2}}}}^{\ell_{ {\mathsf{e}}_{1}}+\ell_{ {\mathsf{e}}_{2}}}
				\lvert
				f_{ {\mathsf{e}}_{1}}(x-\ell_{ {\mathsf{e}}_{2}})
				\rvert
				\mathrm{d} x
				\\
				&= 
				\int_{0}^{x_{0}}{\lvert f_{ {\mathsf{e}}_{1}} \rvert}\mathrm{d} x+\int_{x_0}^{\ell_{ {\mathsf{e}}_{1}}}{\lvert f_{ {\mathsf{e}}_{1}} \rvert}\mathrm{d} x+\ell_{{\mathsf{e}}_{2}}\lvert f_{ {\mathsf{e}}_{1}}(x_{0})\rvert \\
			&\geq \int_{{\mathsf{e}}_{1}} \lvert f_{{\mathsf{e}}_{1}} \rvert \mathrm{d} x+\int_{{\mathsf{e}}_{2}} \lvert f_{{\mathsf{e}}_{2}} \rvert.
		\end{align*}
	Comparing the  P{\'o}lya quotients in Theorem~\ref{varchar} yields the result.
 \end{proof}
In Theorem~\ref{thm:unfolding}, one might attempt to also modify the strength $\tilde{\alpha}_{\hat{\mathsf{v}}}$ at the vertex $\hat{\mathsf{v}}$ of the edge $\hat{\mathsf{e}}$ by setting $\tilde{\alpha}_{\hat{\mathsf{v}}} = \alpha_{\mathsf{v}_{1}} + \cdots + \alpha_{\mathsf{v}_{r}}$. 
However, this will in general no longer increase the the torsional rigidity as the following example demonstrates.

\begin{example}
	\label{exa:no_unfolding}
	For small $\epsilon > 0$, take a metric graph comprised of two intervals of length $\epsilon$ and $1 - \epsilon$, respectively, joined at one vertex $\mathsf{v}_0$ and ending at two vertices $\mathsf{v}_1$, and $\mathsf{v}_2$.
	Set the strength at the vertices $\mathsf{v}_0$ and $\mathsf{v}_1$ to be equal to $1$ and set the strength at $\mathsf{v}_2$ to be zero.
	By unfolding, this graph can be made into an interval of length $1$ with strengths $1$ on both end points, see Figure~\ref{fig:no_unfolding} for an illustration. 
	Its torsion function can be seen to be 
	\[
	\upsilon_{\mathrm{unfolded}}(x) = - \frac{x^2}{2} + \frac{x}{2} + \frac{1}{2}
	\]
	which has integral $\frac{7}{12}$.
	
	On the other hand, for $\epsilon \to 0$, the torsional rigidity of the original graph will converge to  torsional rigidity of an interval of length $1$ with strength $2$ on one side and strength $0$ on the other side. Indeed this can be seen by solving the torsion problem~\eqref{torrob} explicitly and noting the continuous dependence on $\epsilon$. But in this case the torsion function is
	\[
	\upsilon_{\mathrm{limit}}(x)
	=
	- \frac{x^2}{2} + 1
	\]
	which has integral $\frac{5}{6}$.
	Since $\frac{7}{12} < \frac{5}{6}$, torsional rigidities, we conclude that in this case, \emph{unfolding will decrease the torsional rigidity}.
	
\begin{figure}[ht]
		\begin{tikzpicture}[scale = 1.25]

		\begin{scope}			

			\coordinate (A) at (0,0);
			\coordinate (B) at (.5,.5);
			\coordinate (C) at (2.5,0);

%			\draw[fill = black] (A) circle (3pt);
%			\draw[fill = black] (B) circle (3pt);
%			\draw[fill = black] (C) circle (3pt);
			
			\draw (A) -- (B);
			\draw (A) -- (C);
			
			\node[fill, circle, inner sep=2pt, label=below:{$\mathsf{v}_0$}, label=left:{\footnotesize$\alpha_{\mathsf{v}_0} = 1$}] at (A) {};
			\node[fill, circle, inner sep=2pt, label=above:{$\mathsf{v}_1$}, label=right:{\footnotesize$\alpha_{\mathsf{v}_1} = 1$}] at (B) {};
			\node[fill, circle, inner sep=2pt, label=below:{$\mathsf{v}_2$}, label=above:{\footnotesize$\alpha_{\mathsf{v}_2} = 0$}] at (C) {};
		\end{scope}	
		
		\begin{scope}[xshift = 4.5cm]			

			\coordinate (A) at (0,0);
			\coordinate (B) at (3,.5);
			\coordinate (C) at (2.5,0);

%			\draw[fill = black] (A) circle (3pt);
%			\draw[fill = black] (B) circle (3pt);
%			\draw[fill = black] (C) circle (3pt);
			
			\draw (A) -- (C);
			\draw (C) -- (B);
			
			\node[fill, circle, inner sep=2pt, label=above:{$\mathsf{v}_0$}, label=below:{\footnotesize$\alpha_{\mathsf{v}_0} = 1$}] at (A) {};
			\draw[fill = white] (C) circle (2.4pt);
%			\node[fill, circle, inner sep=2pt, label=above:{$\mathsf{v}_1$}, label=right:{\footnotesize$\alpha_{\mathsf{v}_1} = 1$}] at (B) {};
			\node[fill, circle, inner sep=2pt, label=above:{$\hat{\mathsf{v}}$}, label=right:{\footnotesize$\alpha_{\hat{\mathsf{v}}} = 1$}] at (B) {};
		\end{scope}

		\end{tikzpicture}
	\caption{The metric graph and from Example~\ref{exa:no_unfolding} on the left and its unfolded version on the right.}
	\label{fig:no_unfolding}
\end{figure}
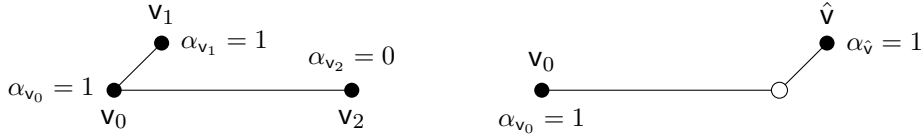		
	
\end{example}

We conclude that there seems no simple answer on how to re-arrange strengths in order to ensure monotonicity of the unfolding process.

\begin{section}{Estimates on torsional rigidity}\label{est}
\begin{subsection}{Lower bound: Flowers are minimizers}
{
In this section, we prove a lower bound on the torsional rigidity in terms of the total length, the number of edges $\lvert \mathsf{E} \rvert$, and the \emph{total strength}, defined as $\lvert \alpha \rvert_1 := \sum_{\mathsf{v} \in \mathsf{V}}\alpha_{\mathsf{v}}$.
We shall see that for fixed total length, total strength and edge number, flower graphs as in Example~\ref{exa:flower} will minimize the torsional rigidity.
}

\begin{theorem}
	Let $\mathcal{G}$ be a connected metric graph and let $\alpha \in [0, \infty)^{\mathsf{V}}$ be not identically zero.
	Then $$ T(\mathcal{G}, \alpha)\geq \frac{\lvert \mathcal{G} \rvert^{3}}{12\lvert \mathsf{E} \rvert^{2}}+\frac{\lvert \mathcal{G} \rvert^{2}}{\lvert \alpha \rvert_1}. $$ 
	In particular at fixed total length, number of edges and total strength, flower graphs as in Example~\ref{exa:flower} minimize the torsional rigidity.
\end{theorem}
	\begin{proof}
		Let $\mathcal{F}$ be the flower graph on $\mathsf{E}$ edges, with the same lengths as for $\mathcal{G}$, and let $\Delta_{\mathcal{F},\lvert \alpha \rvert_1 }$ be the Laplace operator on $\mathcal{F}$ with a $\delta$-vertex condition with strength $\lvert \alpha \rvert_1$ at the only vertex. 
		Clearly, $\mathcal{F}$ and $\Delta_{\mathcal{F}, \lvert \alpha \rvert_1}$ can be obtained from $\mathcal{G}$ and $\Delta_{\mathcal{G}, \alpha}$ by joining all vertices and adding their strengths. 
		Applying Theorem~\ref{gluver} repeatedly, we obtain 
		$$
		T(\mathcal{G}, \alpha)
		\geq 
		T(\mathcal{F}, \lvert \alpha \rvert_1)=\sum_{\mathsf{e} \in \mathsf{E}}\frac{\ell_{\mathsf{e}}^{3}}{12}+\frac{\lvert \mathcal{G} \rvert^{2}}{\lvert \alpha \rvert_1}
		$$ 
		by Example~\ref{exa:flower}. 
		By Jensen's inequality, we deduce $ T(\mathcal{G}, \alpha)\geq \frac{\lvert \mathcal{G} \rvert^{3}}{12\lvert \mathsf{E} \rvert^{2}}+\frac{\lvert \mathcal{G} \rvert^{2}}{\lvert \alpha \rvert_1}$. 	
		On the other hand, any graph with the same number of edges, total length, and total strength as $\mathcal{G}$ can be converted to a flower graph by gluing all its vertices. 
	\end{proof}

\end{subsection}
\begin{subsection}{Upper bounds: Saint-Venant type inequality }
In this subsection, we prove upper bounds on the torsional rigidity. The graph realizing the upper bound will be an interval graph with strength $0$ at one vertex and the sum of all strengths concentrate at the other vertex. 
\begin{theorem}\label{SaintVenant}
	Let $\mathcal{G}$ be a connected metric graph and let $\alpha \in [0, \infty)^{\mathsf{V}}$ be not identically zero.
	Then 
	\begin{equation}\label{SV1}
		T(\mathcal{G}, \alpha)
		\leq 
		T \left({\mathcal{J}}, \{0, \lvert \alpha \rvert_1 \}
		\right)
		=
		\frac{\lvert \mathcal{G} \rvert^{3}}{3}+\frac{\lvert \mathcal{G} \rvert^{2}}{\lvert \alpha \rvert_1},
	\end{equation}
	where $\mathcal{J}$ is the path graph of length $\lvert \mathcal{G} \rvert$ with a $\delta$-vertex condition of strength $\lvert \alpha \rvert_1$ at one vertex and a Neumann condition at the other one. 
	Equality holds if and only if $\mathcal{G}=\mathcal{J}$ with the described vertex conditions.
		\end{theorem}

	\begin{proof}
		Let $\upsilon$ be the torsion function for the Laplace operator $\Delta_{\mathcal{G},\alpha}$. Possibly inserting additional vertices, we may assume that $\upsilon$ takes its maximum and minimum at vertices $\mathsf{x}_{\min}$ and $\mathsf{x}_{\max}$, respectively. 
		Let $\hat{\alpha}_{\mathsf{x}_{\min}}= \lvert \alpha \rvert_1$ and $\hat{\alpha}_{\mathsf{v}}=0$ for all $\mathsf{v}\in \mathsf{V} \cup \{\mathsf{x}_{\max} \}$.
		Then $\upsilon \in H^1(\mathcal{G})$, and
		\begin{eqnarray*}
			T(\mathcal{G},\alpha)&=& \frac{ \left( \int_{\mathcal{G}}\upsilon(x)\mathrm{d} x\right)^{2} }{\int_{\mathcal{G}}\lvert \upsilon'(x) \rvert^{2}\mathrm{d} x+\sum_{\mathsf{v} \in \mathsf{V}}\alpha_{\mathsf{v}}\lvert \upsilon(\mathsf{v}) \rvert^{2}}\\
			&\leq&\frac{\left( \int_{\mathcal{G}}\upsilon(x)\mathrm{d} x\right)^{2} }{\int_{\mathcal{G}}\lvert \upsilon'(x) \rvert^{2}\mathrm{d} x+\lvert \alpha \rvert_1 \cdot \lvert \upsilon(\mathsf{x}_{\min}) \rvert^{2}}	\leq T(\mathcal{G},\hat{\alpha}).
		\end{eqnarray*}
		Using the same symmetrization technique as in~\cite{karr}, we construct a continuous, non-decreasing rearrangement $\upsilon^{*}$ on $[0,\lvert \mathcal{G} \rvert]$ of $\upsilon$ via
		\begin{equation*}
			\upsilon^{*}(0)=\upsilon(\mathsf{x}_{\min}), \ \  \upsilon^{*}(\lvert \mathcal{G} \rvert)=\upsilon(\mathsf{x}_{\max}),  
		\end{equation*}
		and 
		\[
		\mu_{u}(t)
		:= 
		\lvert 
			\{ x\in \mathcal{G}: \upsilon(x)<t\} 
		\rvert
		=
		\lvert 
			\{s\in {[0, \lvert \mathcal{G} \rvert ] }: \upsilon^{*}(s)<t\} .
		\rvert 
		\]
		By construction $\upsilon^{*}$  satisfies
		\begin{equation}\label{eqnnorms}
			\lVert \upsilon \rVert_{L^{p}(\mathcal{G})}^{p}=  \lVert \upsilon^{*}  \rVert_{L^{p}(0,\lvert \mathcal{G} \rvert)}^{p}.
		\end{equation}
		Let $n(t)$ denote the number of preimages of $t$ under $\upsilon$. Since $\upsilon$ is continuous $n(t)\geq 1,$ and $n(t)$ is a finite number, because $\upsilon$ satisfies $-\upsilon_{\mathsf{e}}^{''}=1$ on each interval. 
		Using the coarea formula and Cauchy-Schwarz inequality, we have
		\begin{equation}
			\int_{\mathcal{G}}\lvert \upsilon'(x) \rvert^{2}\mathrm{d} x
			=
			\int_{\upsilon(\mathsf{x}_{\min})}^{\upsilon(\mathsf{x}_{\max})} \sum_{x:\upsilon(x)=t} \lvert \upsilon'(x) \rvert\mathrm{d} t
			\geq 
			\int_{\upsilon(\mathsf{x}_{\min})}^{\upsilon(\mathsf{x}_{\max})} \frac{n(t)^2}{\mu'(t)}\mathrm{d} t.
		\end{equation}
		and 
		\begin{equation}\label{forcor}
			\int_{0}^{\lvert \mathcal{G} \rvert}\lvert \upsilon^{*'}(x)\rvert^{2}\mathrm{d} x
			=
			\int_{\upsilon(\mathsf{x}_{\min})}^{\upsilon(\mathsf{x}_{\max})} 
			\sum_{x:\upsilon^{*}(x)=t} \lvert \upsilon^{*'}(x)\rvert\mathrm{d} t
			=
			\int_{\upsilon(\mathsf{x}_{\min})}^{\upsilon(\mathsf{x}_{\max})} 
			\frac{1}{\mu'(t)}\mathrm{d} t.
		\end{equation}
		Thus, we conclude \begin{equation*}
			 \lVert \upsilon'  \rVert_{L^{2}(\mathcal{G})}^{2}
			\geq
			 \lVert \upsilon^{*'}  \rVert_{L^{2}(0,\lvert \mathcal{G} \rvert)}^{2}, 
		\end{equation*}
		with equality if and only if $n \equiv 1$. 
		It follows that the P{\'o}lya quotients satisfy 
		\begin{eqnarray*}
			T(\mathcal{G},\alpha)&\leq&\frac{ \lVert \upsilon  \rVert_{L^{1}(\mathcal{G})}^{2}}{\int_{\mathcal{G}}\lvert \upsilon'(x) \rvert^{2}+\lvert\alpha\rvert_1 \lvert \upsilon(\mathsf{x}_{\min}) \rvert^{2}}\\
			&\leq& \frac{ \lVert \upsilon^{*} \rVert_{L^{2}(0,\lvert \mathcal{G} \rvert)}^{2}}{ \lVert \upsilon^{*'} \rVert_{L^{2}(0,\lvert \mathcal{G} \rvert)}^{2}+\lvert\alpha\rvert_1 \cdot \lvert \upsilon^{*}(0) \rvert^{2}}
			 \leq  T(\mathcal{J}, \{0, \lvert \alpha \rvert_1 \}),
		\end{eqnarray*}
		with equality if and only if $\mathcal{G}={\mathcal{J}}$. 
		Indeed, if $	T(\mathcal{G},\alpha)= T({\mathcal{J}}, \{0, \lvert \alpha \rvert_1 \})$, then all inequalities must be identities.
		In particular, $n \equiv 1$ whence $\mathcal{G}$ is a path graph.
	\end{proof}

{
\begin{remark}
	In the case of graphs with no $\delta$-vertex conditions, but exactly one Dirichlet vertex, one can give a substantially shortened proof of the first half of the statement of Theorem~\ref{SaintVenant} via surgery by successive applications of cutting and unfolding.
	However, this would not provide a classification of maximizers whence in previous works, authors have resorted to more complicated rearrangement techniques, cf.~\cite[Theorem~4.6]{pluem}.
	In our situation of $\delta$-vertex conditions, proofs via surgery seems inaccessible because unfolding no longer holds as demonstrated in Example~\ref{exa:no_unfolding}.
\end{remark}
}

We will now refine the bound~\eqref{SV1} for a special class of metric graphs: doubly connected graphs. 
\begin{definition}
A metric graph $\mathcal{G}$ is \emph{doubly connected} if for every $x \neq y \in \mathcal{G}$, there are two edge-disjoint paths joining $x$ and $y$.
\end{definition}

Let $\mathcal{G}$ be a metric graph and $\upsilon$ be its torsion function. For $t \in [0, \lVert \upsilon \rVert_{\infty}]$, let $n(t)$ denote the number of preimages of $t$ under $\upsilon$. 
In particular, if $\mathcal{G}$ is doubly connected, then $n(t)\geq 2$, cf~\cite{berk2}. 

\begin{corollary}
	\label{cor:doubly_connected}
	In the situation of Theorem~\ref{SaintVenant} let additionaly $\mathcal{G}$ be doubly connected. 
	Then
	\begin{equation}
		T(\mathcal{G},\alpha)\leq 2 \  T\left( \hat{\mathcal{J}}, \left\lbrace  \frac{\lvert\alpha\rvert_{1}}{2},0\right\rbrace  \right)=\frac{\lvert \mathcal{G} \rvert^{3}}{24}+\frac{\lvert \mathcal{G} \rvert^{2}}{2\lvert\alpha\rvert_{1}}, 
	\end{equation}
	where $\hat{\mathcal{J}}$ is a path graph of length $\frac{\lvert \mathcal{G} \rvert}{2}$.
\end{corollary}

Note that the Laplacian on $\left[ 0,\frac{\lvert \mathcal{G} \rvert}{2}\right] $ with a $\delta$-vertex condition of strength $\lvert\alpha\rvert_{1}/2$ at one end point has half the torsional rigidity of a circle graph of length $\lvert \mathcal{G} \rvert$ with a $\delta$-vertex condition of stength $\alpha$ at one point.	
Thus, one can also interpret Corollary~\ref{cor:doubly_connected} as stating that circle graphs with all delta parameters concentrated at one vertex maximize torsional rigidity among all doubly-connected graphs.
This is a graph analogon of a classic result in~\cite{pol3}, asserting that on non simply connected domains, the torsional rigidity is maximized on annuli.	
	
	\begin{proof}
	Since $\mathcal{G}$ is  doubly connected, we have	$n(t)\geq 2. $
		We let 
		\begin{equation*}
			\hat{\upsilon^{*}}(s)=\upsilon^{*}(2 \lvert s \rvert), \ \ s \in \hat{\mathcal{J}} .
		\end{equation*}
		From equations~\eqref{eqnnorms} and~\eqref{forcor}, we easily deduce
		\begin{equation*}
			\left( 	\int_{0}^{\frac{\lvert \mathcal{G} \rvert}{2}}\hat{\upsilon^{*}}(x)\mathrm{d} x\right)^{2} 
			=\left( \frac{1}{2}\int_{0}^{\lvert \mathcal{G} \rvert} \upsilon^{*}(x) \mathrm{d} x\right)^{2}=\left( \frac{1}{2}\int_{\mathcal{G}} \upsilon(x) \mathrm{d} x\right)^{2},
		\end{equation*}
		\begin{equation*}
			\int_{0}^{\frac{\lvert \mathcal{G} \rvert}{2}}\lvert \hat{\upsilon^{*}}^{'}(x) \rvert^{2}\mathrm{d} x
			=2\int_{0}^{\lvert \mathcal{G} \rvert} \lvert \upsilon^{*'}(x)\rvert^{2} \mathrm{d} x\leq\frac{1}{2}\int_{\mathcal{G}} \lvert \upsilon^{'}(x) \rvert^{2} \mathrm{d} x,
		\end{equation*}
		and
		\begin{align*}
		 T\left( \hat{\mathcal{J}}, \left\lbrace  \frac{\lvert\alpha\rvert_{1}}{2},0\right\rbrace  \right)
		 & =
		 \frac{4\left(	\int_{0}^{\frac{\lvert \mathcal{G} \rvert}{2}}\hat{\upsilon^{*}}(x)\mathrm{d} x\right)^{2} }{4\left(\int_{0}^{\frac{\lvert \mathcal{G} \rvert}{2}}\lvert \hat{\upsilon^{*}}^{'}(x) \rvert^{2}\mathrm{d} x+\frac{\alpha}{2}\hat{\upsilon^{*}}(\mathsf{x}_{\min})^{2}\right) }
		 \\
		 & \leq \frac{\left(\int_{\mathcal{G}} \upsilon(x) \mathrm{d} x \right)^{2} }{2\left( \int_{\mathcal{G}} \lvert \upsilon^{'}(x) \rvert^{2} \mathrm{d} x+\alpha\upsilon(\mathsf{x}_{min})^{2}\right) } =  \frac{T(\mathcal{G},\alpha)}{2}.\qedhere
		\end{align*}
	\end{proof}

\end{subsection}

Finally, we shift our perspective to combine torsional rigidity with the spectrum and and prove an upper bound in terms of torsional rigidity and the lowest eigenvalue $\lambda_1(\mathcal{G}, \alpha)$ of $\Delta_{\mathcal{G}, \alpha}$.

\begin{theorem}\label{polya-szego}
		Let $\mathcal{G}$ be a connected metric graph and let $\alpha \in [0, \infty)^{\mathsf{V}}$ be not identically zero.
	Then 
\begin{equation}
\lambda_{1}(\mathcal{G}, \alpha) T(\mathcal{G}, \alpha)<\lvert \mathcal{G} \rvert.
\end{equation}
\end{theorem}
\begin{proof}
	The proof is analogous to \cite[Section 5.4]{vander} and \cite[Proposition 5.1]{pluem} using the estimate
	\[
	T(\mathcal{G}, \alpha)
	=
	\frac
	{\lVert \upsilon \rVert_{L^1(\mathcal{G})}^2}
	{h_{\alpha}(\upsilon)}
	\leq
	\frac
	{\lvert \mathcal{G} \rvert \cdot \lVert \upsilon \rVert_{L^2(\mathcal{G})}^2}
	{h_{\alpha}(\upsilon)}
	\leq
	\frac{\lvert \mathcal{G} \rvert}{\lambda_1(\mathcal{G}, \alpha)}
	\]
	and sharpness of the Cauchy-Schwarz inequality used therein.
\end{proof}

\begin{subsection}{Discussion on a Kohler-Jobin type inequality}\label{Koh-job}

	A significant achievement in the interplay of torsional rigidity with spectral invariants of two-dimensional planar domains was a precise lower bound on the product of torsional rigidity and the square of the ground-state energy by Kohler-Jobin~\cite{koh}. 
	This bound, initially proposed by P{\'o}lya and Szeg\H{o} in~\cite{vander}, firmly solidified the importance of torsional rigidity in spectral geometry and shape optimization. 
	In~\cite{pluem}, an estimate similar to that of Kohler-Jobin within the context of metric graphs is proven. The proof obtained for the Dirichlet condition in~\cite{pluem} could be deduced in a similar manner to the ideas presented in~\cite{bras,koh}.  The inequality provides the following uniform lower bound:
	\begin{equation}\label{Koh-Job}
		\lambda_{1}(\mathcal{G})T(\mathcal{G})^{2/3}\geq \left( \frac{\pi}{\sqrt[3]{24}}\right)^{2}.
	\end{equation}
	To obtain a Kohler-Jobin inequality for metric graphs with $\delta$-vertex conditions, it is natural to look at interval graphs. Indeed, one of the main steps in the proof is reducing this multiplication $\lambda_{1}(\mathcal{G})T(\mathcal{G})^{2/3}$ from an arbitrary quantum graph $\mathcal{G}$ to an interval graph $\mathcal{J}$.  The uniform lower bound in~\eqref{Koh-Job} is  $\lambda_{1}(\mathcal{J})T(\mathcal{J})^{2/3}.$	However, when attempting to estimate a similar inequality for $\delta$-vertex conditions, the first eigenvalue of an interval graph is not explicit, only given as a root of transcendental equation. This is why, initially, it is not clear how a Kohler-Jobin inequality looks like in the context of $\delta$-vertex conditions. In such a situation, it might make sense to look at lower bounds for the first eigenvalue.
	A lower bound on the first Neumann eigenvalue on intervals with $\delta$-vertex conditions, known to be asymptotically sharp, is  $\lambda_{1}(\mathcal{J})\geq \frac{\pi^{2}\alpha}{\lvert \mathcal{J} \rvert (\pi^{2}+4\alpha \lvert \mathcal{J} \rvert)}$, where $\mathcal{J}$ is an interval graph having Neumann condition at one end point and $\alpha$ strength at the other end point (see \cite[Proposition A1]{ken2}). The product of this lower bound and $T(\mathcal{J})^{2/3}$ is
	 \begin{equation}\label{lowest}
\frac{\pi^{2}\alpha}{\lvert \mathcal{J} \rvert (\pi^{2}+4\alpha \lvert \mathcal{J} \rvert)}\left( \frac{\lvert \mathcal{J} \rvert^3}{3}+\frac{\lvert \mathcal{J} \rvert^{2}}{\alpha} \right)^{2/3}.
	\end{equation}
	This is asymptotically compatible to the case of pure Dirichlet and Kirchhoff-Neumann vertices in \cite[Theorem 5.8]{pluem}.
	\begin{lemma} 
		As $\alpha\to\infty$, the term in~\eqref{lowest} converges to that of~\eqref{Koh-Job}
$$\lim_{\alpha\to\infty} \frac{\pi^{2}\alpha}{\lvert \mathcal{J} \rvert(\pi^{2}+4\alpha\lvert \mathcal{J} \rvert)}\left( \frac{\lvert \mathcal{J} \rvert^3}{3}+\frac{\lvert \mathcal{J} \rvert^{2}}{\alpha} \right)^{2/3}=\left( \frac{\pi}{\sqrt[3]{24}}\right)^{2}.$$
\begin{proof}
Using~\eqref{lowest}, we have
\begin{align*}
	&\lim_{\alpha\to\infty} \frac{\pi^{2}\alpha}{\lvert \mathcal{J} \rvert(\pi^{2}+4\alpha\lvert \mathcal{J} \rvert)}\left( \frac{\lvert \mathcal{J} \rvert^3}{3}+\frac{\lvert \mathcal{J} \rvert^{2}}{\alpha} \right)^{2/3} 
	\\
	= 
	&\lim_{\alpha\to\infty}\frac{\pi^{2}\alpha}{\lvert \mathcal{J} \rvert(\pi^{2}+4\alpha\lvert \mathcal{J} \rvert)} \left( \frac{\alpha \lvert \mathcal{J} \rvert^{3}+3\lvert \mathcal{J} \rvert^{2}}{3\alpha}\right)^{2/3} 
	=
	\left( \frac{\pi}{\sqrt[3]{24}}\right)^{2}.\qedhere
\end{align*}
\end{proof}
	\end{lemma}
	 At this point, there are two possible paths to take: the first one is to constrain the length of the interval graph. As long as $\lvert \mathcal{J} \rvert\geq 1$, the quantity $\lvert \mathcal{J} \rvert$ can be simplified and a uniform lower bound (depending on $\alpha$) can be obtained. 
	 \begin{lemma} Let $\mathcal{J}$ be an interval graph having Neumann condition at one end point and strength $\alpha>0$  at the other end point of length $\lvert \mathcal{J} \rvert\geq 1$, then\begin{equation}
\lambda_{1}(\mathcal{J},\{0, \alpha\})T(\mathcal{J})^{2/3}\geq \frac{\pi^{2}\alpha}{\sqrt[3]{9}(\pi^{2}+4\alpha)}.
\end{equation}
\end{lemma}
The second way would  be to get a more general inequality, i.e. modify the quantity
$\lambda_{1}(\mathcal{J},\{0, \alpha\})T(\mathcal{J},\{0, \alpha\})^{2/3}$ in order to make it independent of the length of the interval $\mathcal{J}$ as in the papers~\cite{bras,koh,pluem}.
 Unfortunately, there is no domain
analog inequality in the context of $\delta$-vertex conditions that could serve as inspiration.

	Even if a meta Kohler-Jobin inequality was conjectured, it turns out that there are more technical challenges. The core part of the proof is symmetrization of functions, that is, constructing a bridge that lets us convert functions on the graph $\mathcal{G}$ to the functions on $\mathcal{J}$  (\cite[Lemma 5.17]{pluem}). But this technique seems not directly adoptable without new ideas to $\delta$-vertex conditions. In summary, it seems not yet completely clear what a Kohler-Jobin inequality on metric graphs in the presence of $\delta$-vertex conditions should even look like and any proof seems to require more than simple modifications of existing results. In light of this, it is unsurprising, that also on domains, with Robin boundary conditions, no Kohler-Jobin inequality is known.

\end{subsection}
\end{section}

\end{document}